\documentclass[12pt,reqno]{article}
mm \textheight=8.9in

\usepackage{amssymb}
\usepackage{amsmath}
\usepackage{amsthm}
\usepackage{color}

\newcommand{\hth}{\mathrm{ht}}

\newcommand{\br}[3]{{$#1$}$\lower4pt\hbox{$\tp\atop\raise4pt \hbox{$\scriptscriptstyle{#2}$}$} ${$#3$}}
\newcommand{\tw}[3]{{$#1$}${\,\scriptscriptstyle {#2}}\atop\raise9pt\hbox{$\scriptstyle\tp$} ${$#3$}}
\newcommand{\ttps}[2]{{#1}\raise5pt\hbox{$\lower12pt\hbox{$\scriptstyle\tp$}\atop \lower0pt\hbox{$\tilde\;$}$}\raise4.5pt\hbox{${\scriptstyle{#2}}$}}
\newcommand{\st}[1]{\mbox{${\,\scriptscriptstyle {#1}}\atop\raise5.5pt\hbox{$*$}$}}

\newcommand{\rd}[1]{\mbox{${\,\scriptscriptstyle {#1}}\atop\raise5.5pt\hbox{$\bullet$}$}}
\newcommand{\rt}[1]{\otimes_\chi}
\newcommand{\lt}[1]{\mbox{${\,\scriptscriptstyle {#1}}\atop\raise5.5pt\hbox{$\ltimes$}$}}
\newcommand{\btr}{\raise1.2pt\hbox{$\scriptstyle\blacktriangleright$}\hspace{2pt}}
\newcommand{\btl}{\raise1.2pt\hbox{$\scriptstyle\blacktriangleleft$}\hspace{2pt}}

\newcommand{\lcr}{\raise1.0pt \hbox{${\scriptstyle\rightharpoonup}$}}
\newcommand{\rcr}{\raise1.0pt \hbox{${\scriptstyle\leftharpoonup}$}}

\newcommand{\ttp}{{\lower12pt\hbox{$\tp$}\atop \hbox{$\tilde\;$}}}

\newcommand{\id}{\mathrm{id}}

\newcommand{\Tc}{\mathcal{T}}

\newcommand{\Fin}{\mathrm{Fin}}
\newcommand{\Sym}{\mathrm{Sym}}

\newcommand{\Ac}{\mathcal{A}}

\newcommand{\A}{{A}}

\newcommand{\J}{J}

\newcommand{\Sc}{\mathcal{S}}

\newcommand{\Ru}{\mathcal{R}}

\newcommand{\Cc}{\mathcal{C}}

\newcommand{\Q}{\mathcal{Q}}
\renewcommand{\O}{\mathcal{O}}

\newcommand{\C}{\mathbb{C}}
\newcommand{\Sbb}{\mathbb{S}}

\newcommand{\Qbb}{\mathbb{Q}}

\newcommand{\Z}{\mathbb{Z}}

\newcommand{\N}{\mathbb{N}}

\newcommand{\tp}{\otimes}

\newcommand{\zt}{\zeta}

\newcommand{\ve}{\varepsilon}
\newcommand{\gm}{\gamma}
\newcommand{\dt}{\delta}

\newcommand{\op}{\oplus}
\newcommand{\la}{\lambda}
\newcommand{\tr}{\triangleright}
\newcommand{\tl}{\triangleleft}

\newcommand{\Char}{\mathrm{ch }}

\newcommand{\Crt}{\C\backslash{\sqrt[\Z]{t}}}
\newcommand{\Crte}{\C\backslash{\sqrt[\Z]{1}}}

\newcommand{\End}{\mathrm{End}}

\newcommand{\Hom}{\mathrm{Hom}}

\newcommand{\Ind}{\mathrm{Ind}}
\newcommand{\rk}{\mathrm{rk}}

\newcommand{\Rm}{\mathrm{R}}

\newcommand{\ad}{\mathrm{ad}}
\newcommand{\Ad}{\mathrm{Ad}}
\newcommand{\La}{\Lambda}

\newcommand{\g}{\mathfrak{g}}

\renewcommand{\b}{\mathfrak{b}}
\renewcommand{\k}{\mathfrak{k}}

\newcommand{\h}{\mathfrak{h}}

\newcommand{\s}{\mathfrak{s}}

\renewcommand{\o}{\mathfrak{o}}

\newcommand{\m}{\mathfrak{m}}
\newcommand{\eps}{\epsilon}

\newcommand{\nn}{\nonumber}
\newcommand{\p}{\mathfrak{p}}
\renewcommand{\l}{\mathfrak{l}}

\newcommand{\Spec}{\mathrm{Spec}}

\newcommand{\si}{\sigma}
\newcommand{\al}{\alpha}

\renewcommand{\b}{\mathfrak{b}}
\newcommand{\bt}{\beta}

\newcommand{\be}{\begin{eqnarray}}
\newcommand{\ee}{\end{eqnarray}}

\newtheorem{thm}{Theorem}[section]
\newtheorem{propn}[thm]{Proposition}
\newtheorem{lemma}[thm]{Lemma}
\newtheorem{corollary}[thm]{Corollary}

\newtheorem{definition}[thm]{Definition}

\newcount\prg

\newcommand{\parag}{\advance\prg by1 {\noindent\bf\thesection.\the\prg\hspace{6pt}}}

\begin{document}
\title{Vector bundles on quantum conjugacy classes}
\author{
Andrey Mudrov \vspace{10pt}\\
\small University of Leicester, \\
\small University Road,
LE1 7RH Leicester, UK,
\vspace{10pt}\\
\small
 Moscow Institute of Physics and Technology,\\
\small
9 Institutskiy per., Dolgoprudny, Moscow Region,
141701, Russia,
\vspace{10pt}\\
\small e-mail:  am405@le.ac.uk}

\date{ }

\maketitle

\begin{abstract}
Let  $\mathfrak{g}$ be a simple complex Lie algebra of a classical type and  $U_q(\mathfrak{g})$ the
corresponding Drinfeld-Jimbo quantum group at $q$ not a root of unity.
With every point $t$ of the fixed maximal torus $T$ of an algebraic group $G$ with Lie algebra $\mathfrak{g}$ we associate an additive category $\mathcal{O}_q(t)$ of
$U_q(\mathfrak{g})$-modules that is stable under tensor product  with finite-dimensional quasi-classical
$U_q(\mathfrak{g})$-modules.
We prove that
 $\mathcal{O}_q(t)$ is essentially semi-simple and use it to explicitly quantize equivariant vector
 bundles on the conjugacy class of $t$.
\end{abstract}

\\
{\small \underline{AMS classification codes}: 17B10, 17B37, 53D55.}
 \newpage
\tableofcontents
\newpage
\section{Introduction}
This paper is devoted to  quantization of  the category of equivariant vector bundles on
a semi-simple conjugacy class of a simple complex algebraic group $G$. This includes quantization
of the function algebra as a trivial bundle of rank 1.
This work is a continuation of a project started off in \cite{M4,M5,JM} and  technically based on \cite{M1,M2}.
 A complete analysis is done for groups of the four infinite series.
With regard to five exceptional types, we believe that the approach is generally applicable  as well.
The main technical issue to address is the quasi-classical behaviour of Shapovalov elements, which is sorted
out for the classical types in this paper.
 
Semi-simple are the only conjugacy classes that  are affine sub-varieties
in $G$ \cite{Spr}.
By a vector bundle we understand a projective module of global sections over the coordinate ring,
 in accordance with the Serre-Swan theorem, \cite{S,Sw}.
We adopt this point of view in  the quantum setting and treat vector bundles over a non-commutative space
as projective (one-sided) modules over its quantized coordinate ring. This way the deformation quantization programme for Poisson
varieties naturally extends to the realm of vector bundles.
In the presence of symmetry, it essentially becomes a part of representation theory.

The  Poisson structure underlying the quantization of our concern descents from the Semenov-Tian-Shansky bracket on $G$
related to the standard classical $r$-matrix, \cite{STS}. It makes $G$ a Poisson variety over the Poisson group $G$ with
the Drinfeld-Sklyanin bracket generated by $r$, with respect to the conjugation action. The non-trivial Poisson structure on $G$
implies quantization
of the symmetry
group first. Equivariance is then understood relative to the quantized universal enveloping algebra $U_q(\g)$
of the Lie algebra $\g=\mathrm{Lie}(G)$.

The Semenov-Tian-Shansky bracket on $G$ is analogous to the $G$-invariant Lie bracket on the Lie algebra $\g\simeq \g^*$, which restricts to
every adjoint orbit. Equivariant quantization of semi-simple orbits in $\g$ is of long interest and has been understood some twenty years ago
\cite{DGS,DM,EE,AL}, in what concerns the conventional point of view restricted to function algebras.
The underlying representation theory involves  parabolic Verma modules over the classical universal enveloping algebra $U(\g)$.

A representation theoretical approach to equivariant quantization consists in realization of the quantized function algebra on a $G$-space
by linear operators on a  $U_q(\g)$-module (respectively $U(\g)$-module in the case of a $\g$-orbit). It is natural
to seek for a realization of a more general quantum vector bundle  via
linear mappings between modules from an appropriate category. They generalize parabolic modules
from the Bernstein-Gelfand-Gelfand  category  of $U_q(\g)$-modules, which we denote by $\O_q$.  Such modules
form an additive subcategory in $\O_q$ determined, up to an isomorphism, by a point $t$ from a fixed maximal torus $T\subset G$.
Modules of highest weight in it are parameterized by finite-dimensional irreducible representations of the subalgebra $\k\subset \g$ centralizing $t$.

  The category under study is stable under tensor product with finite dimensional $U_q(\g)$-modules. It is generated by a base module $M_\la$ of highest weight $\la$ associated
 with $t$ and denoted by $\O_q(t)$.
 The base weight $\la$ is not uniquely determined by $t$ but up to an action of the group of $U_q(\g)$-characters
 known to be $\simeq \Z_2^{\rk\> \g}$, where $\Z_2=\Z/2\Z$. This group acts by isomorphisms on the
 categories associated with $t$ via
tensor product with the corresponding one-dimensional module.

The base module $M_\la$ supports quantization of the coordinate
ring $\C[O]$ of the conjugacy class $O\ni t$ as a subalgebra in $\End(M_\la)$. At least for all non-exceptional $G$, different $t$ give rise to
 isomorphic quantizations of $\C[O]$ but different
faithful representations,  cf. e.g. \cite{AM}. We expect that  be true for all types of $G$.

If $t$ is of finite order, then $\O_q(t)$ is
semi-simple for each $q$ not a root of unity except maybe for a finite set of values.
In all cases that we worked out explicitly \cite{M5,JM}, the set of exceptional $q$ is empty.
For general $t$, the category $\O_q(t)$ is semi-simple for almost all $q$ away from
the roots of the spectrum of the adjoint operator $\Ad_t\in \End(\g)$.
 
 The category  $\O_q(t)$  proves to be equivalent to the category of equivariant finitely generated projective modules
over the quantized polynomial ring $\C[O]$, as a module category over the finite-dimensional
quasi-classical representations of $U_q(\g)$. As an Abelian category, $\O_q(t)$ is equivalent to that
of classical $\k$-modules which are submodules in finite-dimensional $\g$-modules.

In the final section, we construct an equivariant star product on $\C[O]$ by twisting the multiplication on the
RTT algebra of  functions on the quantum group, \cite{FRT}.
Its restriction to a space of "$\k$-invariants" delivers a flat associative  deformation of $\C[O]$.
This construction is not new for $\k$ of Levi type. For a non-Levi $\k$,
it was done only for  even quantum spheres in \cite{M3}, with the use of elementary harmonic analysis on
the quantum Euclidean space.
We further extend this star product to
associated vector bundles along the lines of \cite{DM}.
Following an approach of \cite{M2} we explicitly express it
through the extremal projector of $U_q(\g)$.

Our main technical tool is contravariant form on $U_q(\g)$-modules and its relation with extremal projector.
Such forms appear in this theory in a few incarnations.

First of all, we use the contravariant form on  Verma modules to construct their generalized parabolic quotients.
Matrix entries of the inverse form constitute  Shapovalov elements $\phi_{m\al}\in U_q(\g_-)$ of weight $-m\al$ for
a simple positive root  $\al$ of $\k$ and $m\in \N$. Applied to the highest vector, $\phi_{m\al}$ produce  extremal vectors
in Verma modules that vanish in the generalized parabolic quotients.
We require   that $\phi_{m\al}$ turns into the power $f_\al^m$ of the root vector $f_\al\in \k_-$ in the classical limit $q\to 1$.
We check it via a direct analysis of matrix elements of the inverse Shapovalov form  \cite{M6} using a
factorization of Shapovalov elements  \cite{M8}.

Another application of contravariant forms is a proof of irreducibility of the base module $M_\la$.
We approximate its opposite module $M_\la'$ of lowest weight $-\la$ by
a sequence of $U_q(\g_+)$-submodules in a certain system $\Xi$ of finite-dimensional $U_q(\g)$-modules,
using the extremal projector.
The set $\Xi$ comprises all counterparts of classical $G$-modules $V$ with an orbit isomorphic to $O$ (exactly those appearing in
$\C[O]$). They admit  $U_q(\g_+)$-homomorphisms $M_\la'\to V$ whose common kernel over all $V\in \Xi$ is zero.
Then we approximate the inverse invariant pairing between $M_\la$ and $M_\la'$ by
certain extremal vectors in $V\tp M_\la$ when $V$ ranges in $\Xi$.
An extremal vector determines a map $M_\la \to  V$ whose injectivity on certain spaces
is equivalent to irreducibility of $V$.

The third appearance of contravariant forms in this presentation is a proof of complete reducibility of tensor products.
From  the base module we proceed to the category $\O_q(t)$  it generates.
It is found in \cite{M1} that  a contravariant form controls complete reducibility of tensor products of modules
of highest weights. A relation between the  form and extremal projector established in \cite{M2} delivers a practical computational machinery
which helps us prove that all modules from the category under study are semi-simple for almost all $q$. The simple objects are generalized parabolic Verma modules $M_{\la,\xi}$ of highest weight $\la+\xi$, where
$\xi$ is a highest weight of a $\k$-submodule in a finite dimensional $\g$-module.

Finally,  inverted contravariant forms participate in definition of the star product on $\C[O]$ and its actions on vector bundles,
like in \cite{AL,DM,EE,EEM,KST}.

We prove that the locally finite part of the $U_q(\g)$-module $\End(M_{\la})$ is a quantization
of $\C[O]$, for almost all $q$.
An irreducible decomposition of  $V\tp M_\la\in \O_q(\la)$ gives rise to a direct sum decomposition of  $V\tp \End(M_\la)$ making
the locally finite part of  $\Hom   (M_{\la,\xi},M_\la)$ (respectively, $\Hom(M_\la,M_{\la,\xi})$)
 a quantization of the equivariant vector bundle with the $\k$-submodule of highest weight $\xi$ in $V$ (respectively, its dual) as the fiber.
The direct sum decomposition of  $V\tp \End(M_\la)$ is quasi-classical and goes to the decomposition of the trivial vector bundle $V\tp \C[O]$
into a sum of equivariant sub-bundles.

\section{Preliminaries}
\label{SecPrelim}

Throughout the paper we assume that the deformation parameter $q$ takes values in the set  $\Crte$ of non-zero complex numbers which are not
a root of unity.
We introduce topology on $\Crte$ as induced from Zariski topology on $\C$. That is, an open set in $\Crte$ is
the complement to a finite set of points, possibly empty.
By all $q$ we understand all from $\Crte$, and almost all means all from a non-empty Zariski open set.

We mostly work over the ground filed $\C$ but some topics
require consideration  over  $\C[q,q^{-1}]$ and further  extension to the local ring $\C_1(q)$ of rational functions in $q$
regular at the classical point $q=1$.

By deformation of a complex vector space $A$ we mean an arbitrary $\C[q,q^{-1}]$-module $A_q$ such that $A_q/(q-1)A_q\simeq A$. We call it flat if,
upon extension over $\C_1(q)$, $A_q\simeq A\tp \C_1(q)$.
By quantization of $A$ we understand its flat deformation, along with additional structures, e.g. algebras, modules {\em etc}.
Such a structure is preserved by a quantum group in equivariant quantization.
\subsection{Quantum group basics}
Let  $\g$ be a simple complex Lie algebra of the classical type and  $\h\subset \g$ its Cartan subalgebra. Fix
a triangular decomposition  $\g=\g_-\op \h\op \g_+$  with  maximal nilpotent Lie subalgebras
$\g_\pm$.
Denote by  $\Rm=\Rm_\g$ the root system of $\g$, and by $\Rm^+=\Rm^+_\g$ the subset of positive roots with basis $\Pi=\Pi_\g$
of simple roots.  The weight lattice is denoted by  $\La=\La_\g$  and the semi-group of dominant weights by $\La^+=\La^+_\g$.
We use a similar notation for reductive subalgebras in  $\g$
and drop the subscript  when only the total Lie algebra $\g$ is in the context.

Choose an inner product $(.,.)$ on $\h$ as a multiple of a restricted $\ad$-invariant form
and transfer it to the  space $\h^*$ of linear functions on $\h$ by duality. 
For every $\la\in \h^*$  denote by  $h_\la \in \h$  a unique element such that $\mu(h_\la)=(\mu,\la)$, for all $\mu\in \h^*$.

By $U_q(\g)$ we understand the standard quantum group \cite{D1,ChP} as a complex Hopf algebra with the set of generators $e_\al$, $f_\al$, and
$q^{\pm h_\al}$ labeled with a simple root $\al$ and satisfying relations
$$
q^{h_\al}e_\bt=q^{ (\al,\bt)}e_\bt q^{ h_\al},
\quad
[e_\al,f_\bt]=\dt_{\al,\bt}[h_\al]_q,
\quad
q^{ h_\al}f_\bt=q^{-(\al,\bt)}f_\bt q^{ h_\al},\quad \forall \al, \bt \in \Pi.
$$
The elements $q^{h_\al}$ are assumed invertible, with $q^{h_\al}q^{-h_\al}=1$, while  $\{e_\al\}_{\al\in \Pi}$ and $\{f_\al\}_{\al\in \Pi}$
also satisfy quantized Serre relations, see \cite{ChP} for details. Here and
throughout the text we use the notation $[z]_q=\frac{q^{z}-q^{-z }}{q-q^{-1}}$ for $z\in \h+\C$.
The complex number $q\not =0$ is not a root of unity.

We fix the Hopf algebra structure on $U_q(\g)$ by setting  comultiplication on the generators as
$$
\Delta(f_\al)= f_\al\tp 1+q^{-h_\al}\tp f_\al,\quad \Delta(q^{\pm h_\al})=q^{\pm h_\al}\tp q^{\pm h_\al},\quad\Delta(e_\al)= e_\al\tp q^{h_\al}+1\tp e_\al.
$$
Then the antipode acts on the generators by the assignment
$$
\gm( f_\al)=- q^{h_\al}f_\al, \quad \gm( q^{\pm h_\al})=q^{\mp h_\al}, \quad \gm( e_\al)=- e_\al q^{-h_\al}.
$$
and the counit returns
$$
\eps(e_\al)=0, \quad \eps(f_\al)=0, \quad \eps(q^{h_\al})=1.
$$

We denote by $U_q(\h)$,  $U_q(\g_+)$, and $U_q(\g_-)$  the associative unital subalgebras in $U_q(\g)$  generated by $\{q^{\pm h_\al}\}_{\al\in \Pi}$, $\{e_\al\}_{\al\in \Pi}$, and $\{f_\al\}_{\al\in \Pi}$, respectively.
The quantum Borel subgroups are defined as $U_q(\b_\pm)=U_q(\g_\pm)U_q(\h)$; they are Hopf subalgebras in $U_q(\g)$.

We consider an involutive coalgebra anti-automorphism and algebra automorphism $\si$ of $U_q(\g)$ setting it on the
generators by the assignment
$$\si\colon e_\al\mapsto f_\al, \quad\si\colon f_\al\mapsto e_\al, \quad \si\colon q^{h_\al}\mapsto q^{-h_\al}.$$
The involution $\omega =\gamma^{-1}\circ \si=\si \circ \gm$ is an algebra anti-automorphism and preserves
comultiplication.

With every  normal order on $\Rm^+$ (a sum of two positive roots is between the summands) one associates
a Lusztig system of root vectors $f_\al,e_\al$, for all $\al \in \Rm^+$.
Every such pair defines an associative subalgebra $U_q(\g^\al)\subset U_q(\g)$ that is isomorphic to $U_q\bigl(\s\l(2)\bigr)$.
Ordered monomials in $e_\al$ and $f_\al$ deliver a Poincare-Birkhoff-Witt (PBW) basis in $U_q(\g_+)$
and in $U_q(\g_-)$, respectively, \cite{ChP}.

By $G$ we denote a simple algebraic group with Lie algebra $\g$. Unless it is explicitly specified,
it can be any group between simply connected and the adjoint group.
Let $T\subset G$ denote the maximal torus in $G$ whose Lie algebra is $\h$.
We denote by $T_\Qbb\subset T$ the subset of
elements of finite order, i.e. $t\in T_\Qbb$ if and only if $t^m=1$ for some integer $m$.

For each $t\in T$ we call its centralizer $\k\subset \g$ generalized Levi subalgebra.
The polarization of $\g$ induces a  polarization $\k=\k_-\op \h\op \k_+$ such that
$\k_\pm\subset \g_\pm$. The subalgebra $\k\subset \g$  is called Levi if $\Pi_\k\subset \Pi_\g$.
We call it pseudo-Levi if it is not isomorphic to a Levi subalgebra via an internal isomorphism,
for example, if $\k\varsubsetneq \g$ is semi-simple.
This terminology is compatible with what is accepted in the literature, see e.g. \cite{Cost}.
In general, even  for $t$   with a Levi centralizer  there may be other points in  $T$ from the same conjugacy class whose centralizers are not Levi in the above sense.
The Levi type of $\k$ is special  because $U(\k)$ is quantizable as a Hopf subalgebra $U_q(\k)\subset U_q(\g)$.

Given a $U_q(\g)$-module $Z$ we denote by $Z[\mu]$ the  subspace of weight $\mu\in \h^*$,
 i.e. the set of vectors $z\in Z$ satisfying
$q^{h_\al}z=q^{(\mu,\al)}z$ for all $\al \in \Pi$.
The set of weights of $Z$ is denoted by $\La(Z)$, which notation is also used for general $U_q(\h)$-modules.
All modules are assumed $U_q(\h)$-diagonalizable with finite dimensional weight spaces or locally finite over $U_q(\g)$
with finite dimensional isotypic components.
For such a module $Z$, the (right or left) restricted dual is denoted by $Z^*$. If $Z$ is a module of highest weight $\la$,
its opposite module of lowest weight $-\la$ is denoted by $Z'$. There is a linear bijection $\hat \si: Z\to Z'$
intertwining  the representation homomorphisms $\pi$ and $\pi'$ via the involution $\si$:
$\hat \si\circ \pi(x) = \pi'(x)\circ \si$ for
all $x\in U_q(\g)$.

For a diagonalizable $U_q(\h)$-module $V$ with finite dimensional weight spaces we
define infinitesimal character
as a formal sum
$
\sum_{\mu\in \La(V)} \dim V[\mu]_\mu e^{\mu}.
$
We write $\Char(V)\leqslant \Char(W)$ if $\dim V[\mu]\leqslant \dim W[\mu]$ for all $\mu$
and $\Char(V)< \Char(W)$ if the inequality is strict for some $\mu$.

One-dimensional representations of $U_q(\g)$ are trivial on $U_q(\g_\pm)$ and assign $\pm 1$ to every $q^{h_\al}$, $\al\in \Pi$.
They form a group of characters isomorphic to $\Z_2^{\rk \g}$.
The category of quasi-classical finite dimensional $U_q(\g)$-modules is denoted by $\Fin_q(\g)$.
Such modules are diagonalizable with weights from $q^\La$.
General finite dimensional $U_q(\g)$-modules are obtained from  $\Fin_q(\g)$ by tensoring with a one dimensional
module.

\subsection{Poisson Lie structure on conjugacy classes}
In this section we recall the Poisson-Lie structure on the group $G$ that is a Poisson-Lie
analog of the Kostant-Kirillov-Souriau bracket on the (dual of) Lie algebra $\g$.

Fix the $\ad$-invariant inner product on $\g$.
Let $f_\al$ and $e_\al$, $\al \in \Pi_\g$ be the Chevalley generators of the Lie algebra $\g$
satisfying $(e_\al,
f_\al)=1$ and let $\{h_i\}_i$ be an orthogonal basis in $\h$.
The element
$$
r=\sum_i h_i\tp h_i + \sum_{\al \in \Rm^+} e_\al\tp f_\al \in \g\tp \g
$$
is called classical r-matrix. It satisfies the classical Yang-Baxter equation, cf. \cite{D1}.

Consider the left and right invariant vector fields on the group $G$,
$$
\xi^lf(g)=\frac{d}{dt}f(ge^{t\xi})|_{t=0}, \quad \xi^rf(g)=\frac{d}{dt}f(e^{t\xi}g)|_{t=0},
$$
generated by $\xi \in \g$,  where $f\in \C[G]$ and  $g\in G$.
The bivector field $r^{l,l}-r^{r,r}$ makes $G$ a Poisson group, cf. \cite{D3}.

For each  $\xi\in \g$, let $\xi^\ad$ denote the vector field $\xi^l-\xi^r$ on the group $G$.
Put $r_\pm=\frac{1}{2}(r_{12}\pm r_{21})$ to be the symmetric and skew symmetric parts of $r$.
The Semenov-Tian-Shansky (STS)  bivector field
\be
r_-^{\ad,\ad} +(r_+^{r,l}-r_+^{l,r})
\label{STSbr}
\ee
on $G$
is a Poisson structure, \cite{STS}.
It makes $G$ a Poisson-Lie manifold over the Poisson group $G$ under the conjugation action.
This bivector field is tangent to every conjugacy class making it a homogeneous Poisson-Lie manifold over $G$, \cite{AlM}.

Quantization of the STS bracket gives rise to an algebra $\C_q[G]$ satisfying  reflection equation \cite{KS},
and a semi-simple conjugacy class can be quantized as a quotient of  $\C_q[G]$, provided certain
technical conditions are fulfilled. This point of view was developed, e.g., in \cite{M7}.
In this paper, we view $O$ as a quotient space $G/K$, where the subgroup $K\subset G$
is the centralizer of $t$. We construct a local star product on sections of equivariant vector
bundles on $O=\Ad_G(t)$ in the spirit of \cite{DM}.

Let us describe the restriction of the STS bracket to the
class $O$ of a semi-simple element $t\in G$.
The Lie algebra $\g$ splits into the direct sum
$\g=\k\oplus\m$ of vector spaces, where
 $\m$ is the $\Ad_t$-invariant subspace where $\Ad_t-\id$ is invertible.
This decomposition is orthogonal with respect to the $\ad$-invariant form on $\g$,
and $\m$ splits to direct sum $\m_-\op \m_+$ of mutually dual subspaces $\m_\pm=\m\cap \g_\pm$.

The tangent space to $O$ at the point $t$ is naturally identified with $\m$ via the action of $G$.
Choose a basis $\{e_\mu\}\subset \m$   of root vectors. We have $(e_\mu,e_\nu)=0$
unless $\mu +\nu=0$ and assume the normalization $(e_\mu,e_{-\mu})=1$.
The restriction of the Poisson bivector (\ref{STSbr}) to tangent space at the point $t$
is the bivector
$$r_{\m\wedge \m}+ \sum_{\mu \in \Rm_{\g/ \k}}
\frac{\mu(t)+1}{\mu(t)-1}
e_\mu \tp
e_{-\mu}\in \m\wedge \m,
$$
where the first term is the orthogonal projection of $r$ to $\m\wedge \m$.
The second term is correctly
defined since $\Ad_t-\id$  is invertible on $\m$.
\subsection{Definition of generalized parabolic Verma modules.}
\label{Sec-GPVM-def}
In this section we introduce the main object of our study: a class of $U_q(\g)$-modules
that generalize parabolic Verma modules.
We postpone a detailed study of their properties to Section 5 because we need
a certain machinery   that we develop in Sections 3 and 4.

 Every root $\al$ is a (multiplicative) character on $T$ returning
the eigenvalue of the operator $\Ad_t$ on the $\al$-root subspace in $\g$.
By definition of centralizer subalgebra, we have $\al(t) =1$  if and only if $\al\in \Rm_\k$.
For  $t\in T_\Qbb$ and  $\al\in \Rm^+_{\g/\k}$,  the value  $\al(t)\not =1$ is a complex root of unity.

Let $\kappa\in \h^*$ designate the  half-sum of positive roots of  $\k$.
\begin{definition}
  We call  $\la\in \h^*$ a base weight associated with $t\in T$ if
\be
q^{(\la+\rho,\al^\vee)}_\al=\pm\sqrt{\al(t)}q^{(\kappa,\al^\vee)}_\al, \quad \forall \al \in \Pi_\g.
\label{def_base weight}
\ee
\end{definition}
Here by $\sqrt{\al(t)}$ we mean one of two square roots of $\al(t)\in \C^\times$.
Thus the point $t$ does not determine a multiplicative base weight uniquely but up to the sign in $\pm \sqrt{\al(t)}$ for each $\al\in \Pi_\g$.

Recall that an assignment  $q^{h_\al}\mapsto \pm 1$, $f_\al\mapsto 0$, $e_\al\mapsto 0$ for each $\al \in \Pi_\g$ defines a one-dimensional representation of $U_q(\g)$.
They form a group of $U_q(\g)$-characters that is isomorphic to $\Z_2^{\rk \>\g}$.
This group freely acts on the set of base weights of the same $t$.
Each base weight will label a category of generalized parabolic modules of our interest.
The group of $U_q(\g)$-characters acts on those categories by isomorphisms via  tensoring with the corresponding
one-dimensional $U_q(\g)$-modules. Thus one can think that $\la$ has been fixed for each $t$ in what follows.

By $\La_\k$ we denote the weight lattice of the semi-simple part of $\k$.
Since $\Rm_\k\subset \Rm_\g\subset \h^*$ and the canonical form on $\h^*$ is non-degenerate, we consider $\La_\k$ as a subset in $\h^*$.
A base weight $\la$ generates an affine shift $\la+\La^+_\k\subset \h^*$ of the semi-lattice $\La^+_\k$ of $\k$-dominant weights.
Elements from $\la+\La^+_\k$ will be highest weights of the modules of our concern.
The set $\La^+_\k\cap \La_\g$ labels  irreducible equivariant vector bundles on  $O$, thanks to the
Frobenius reciprocity.

For each $\xi \in \La^+_\k$ consider a character of the algebra $U_q(\h)\simeq \C[T]$
by the assignment
\be
q^{h_\al}\to q^{(\la+\xi,\al)}=\pm\sqrt{\al(t)}q^{2(\kappa- \rho+\xi, \al)}, \quad \forall \al \in \Pi_\g,
\label{base weight}
\ee
where the signs have been fixed with the choice of $\la$. Note with care that
we use  exponential presentation for the Cartan generators  for computational convenience.
It means that  $q^{(\la+\xi,\al)}=q^{(\la+\xi,\al)}$ is polynomial in $q$.
The weight $\la+\xi$ satisfies a Kac-Kazhdan condition
\be
[(\la+\xi+\rho,\al^\vee)-m_\al]_{q_\al}=0, \quad m_\al=(\xi,\al^\vee)+1, \quad \forall \al \in \Pi_\k.
\label{Kac-Kazhdan}
\ee
Therefore the Verma  module $\tilde M_{\la+\xi}$ of highest weight $\la+\xi$ has Verma submodules of highest weights
$
\la+\xi-m_\al \al
$
for each $\al \in \Pi_\k$ \cite{DCK}.
\begin{definition}
The quotient of $\tilde M_{\la+\xi}$ by $\sum_{\al\in \Pi_\k}\tilde M_{\la+\xi-m_\al \al}$
is called generalized parabolic Verma module and denoted by $M_{\la,\xi}$.
\end{definition}

Note that if $\k$ is a Levi subalgebra in $\g$ relative to the fixed triangular decomposition,
then $M_{\la,\xi}$ is a parabolic Verma module induced from
$\C_\la\tp X_\xi$, where $X_\xi$ is the finite dimensional $U_q(\k)$-module of highest weight $\xi$
and $\C_\la$ is the one dimensional $U_q(\k)$-module of weight $\la$. The parabolic case is well studied,
and the  most interesting situation is when $\Pi_\k\not \subset \Pi_\g$.

In the special case of $\xi=0$, we denote $M_\la=M_{\la,0}$ and call it base module.
We expect
that $\Char M_\la$ equals the character of the polynomial algebra $\C[\g_+/\k_-]$ up to the factor $e^\la$.
Upon identification $\h^*\simeq \h$ via the inner product on $\h^*$, one can think of
$$
q^{2\la}=t q^{2\kappa-2\rho}\in T
$$
 as a quantization of the initial point $t$.

The set of eigenvalues  $\{\al(t)\}_{\al\in \Rm}\cup\{1\}$ of the operator $\Ad_t\in \End(\g)$ is an invariant of the class $O\in t$.
We call it spectrum of the class/point $t$. Introduce a notation 
\be
\sqrt[\Z]{t}=\{q\in \C^\times|\> q_\al^{m}=\al(t), \>\forall m\in \Z,\al\in \Rm^+\}.
\ee
This set comprises roots of units and roots of  $\bigl(\al(t)\bigr)^{\frac{2}{(\al,\al)}}$ for all $\al\in \Rm^+_{\g/\k}$.
Clearly $\sqrt[\Z]{t}$  depends only on the class of $t$.
It will play a role of the exceptional set for the deformation parameter where the properties of 
 modules participating in quantization of  the class $O\supset t$ may be violated.

\subsection{Extremal  projector}
In this section we recall the $q$-version of extremal projector,  \cite{AST,KT}, which is the key instrument for this study.
We start with the case of $\g=\s\l(2)$ and normalize the inner product so that  $(\al,\al)=2$ for its only positive root $\al$.
Set $e=e_\al$, $f=f_\al$, and $q^{h}=q^{h_\al}$ to be the standard generators of $U_q(\g)$.
Extend $U_q(\g)$ to $\hat U_q(\g)$ by including infinite sums of elements from $\C[f]\C[e]$ of same weights
with coefficients in the field of fractions $\C(q^{\pm h})$.
Similar extension works for general semi-simple $\g$ resulting in
an associative algebra $\hat U_q(\g)$, see e.g. \cite{KT}.

Define $p(s)$ as a rational trigonometric function of $s\in \C$ with values in $\hat U_q\bigl(\s\l(2)\bigl)$:
\be
\label{translationed_proj}
p(s)=\sum_{k=0}^\infty  f^k e^k \frac{(-1)^{k}q^{k(s-1)}}{[k]_q!\prod_{i=1}^{k}[h+s+i]_q}.
\ee
It is stable under the involution $\omega$.

For every module $V$ with locally nilpotent action of the generator $e$,  the element $p(s)$ delivers  a rational trigonometric
endomorphism of every weight space. On a module of highest weight
$\la$,
it acts by
\be
\label{proj_eigen}
p(s)v=c\prod_{k=1}^{l}\frac{[s-k]_q}{[s+\eta(h)+k]_q}v,
\ee
where $v$ is a vector of weight $\eta=\la-l\al$ and $c=q^{-l \eta(h)-l(l+1)}\not =0$.

Consider the truncated operator
$$
p_m(s)=\sum_{k=0}^m  f^k e^k \frac{(-1)^{k}q^{k(s-1)}}{[k]_q!\prod_{i=1}^{k}[h+s+i]_q}.
$$
\begin{lemma}
\label{cut_proj}
Suppose that $V$ is  a $U_q(\g)$-module and a weight subspace $V[\mu]$ is killed by $e^{m+1}$.
If
$$
[(\mu,\al)+i]_q\not =0, \quad \mbox{for}\quad i=0, \ldots, m+1,
$$
then $p_m(1)V[\mu]\subset\ker(e)$.
\end{lemma}
\begin{proof}
The operator $p_m(s)$
satisfies the relation
\be
ep_m(s)=\frac{q^{-h-2}[s-1]_q}{[h+s-1]_q} p_{m-1}(s-1)e+  f^{m} e^{m+1}\frac{(-1)^{m}q^{m(s-1)}}{[m]_q!\prod_{i=1}^{m}[h+s+i]_q}
\label{proj_cut}
\ee
which implies the statement.
\end{proof}

For general $\g$ fix a normal order on $\Rm^+$ and consider an embedding $\iota_\al\colon \hat U_q\bigl(\s\l(2)\bigr)\to \hat U_q(\g)$
defined by the Lusztig pair of root vectors $f_\al, e_\al$, for every positive root $\al\in \Rm^+$.
Let  $p_\al(s)$ denote the image of $p(s)$
in $\hat U_q(\g)$ under $\iota_\al$.
Put  $\zt_i=2\frac{(\zt,\al^i)}{(\al^i,\al^i)}\in \C$ for $\zt\in \h^*$ and $\al^i\in \Rm^+$ and define
\be
p_\g(\zt)=p_{\al^1}(\rho_1+\zt_1)\cdots p_{\al^n}(\rho_n+\zt_n),\quad n=\#\Rm^+,
\label{factorization}
\ee
assuming the product ordered over increasing positive roots.
It is independent of the normal ordering and turns to the extremal projector $p_\g$  at $\zt=0$,
which  is the only element of zero weight
from  $1+\g_-\hat U_q(\g) \g_+$ satisfying
$$
p_\g^2=p_\g, \quad e_\al p_\g =0 =p_\g f_\al , \quad \forall \al \in \Pi.
$$
Uniqueness implies that $p_\g$ is $\omega$-invariant.
\begin{propn}
  For all $\zt\in \h^*$, the operator $p_\g(\zt)$ is $\omega$-invariant.
\label{shifted-proj-invar}
\end{propn}
\begin{proof}
  It is sufficient to prove that $p_\g(\zt)$ is $\omega$-invariant as an operator on every finite dimensional $U_q(\g)$-module $V$, for
  generic $\zt$. Choose $\zt$ such that for all $\mu\in \La(V)$ and all $\al\in \Rm^+_\g$, $(\zt+\mu,\al^\vee)\not \in -\N$.
  Let $Z$ be the Verma module of highest weight $\zt$.
  The projector $p_\g$ is well defined as a linear map from $V\tp 1_\zt$ to
  the space of $U_q(\g_+)$-invariants in $V\tp\tilde  M_\zt$. Then for all $v,w\in V$ the matrix
  element $\bigl(p_\g(\zt) w,v\bigr)$ equals
   $$
   \bigl(p_\g(\zt) w,v\bigr)=\bigl(p_\g (w\tp 1_Z),v\tp 1_Z\bigr)=\bigl(w\tp 1_Z,p_\g (v\tp 1_Z)\bigr)
   = \bigl(w,(p_\g(\zt)b\bigr), $$
  as required. The left and right equalities are due to \cite{M2}, Proposition 3.1. The middle equality employs $\omega$-invariance of the extremal projector.
\end{proof}

\section{Tensor product of highest weight modules}
A key issue arising in equivariant quantization of semi-simple conjugacy classes is semi-simplicity of  certain  tensor product modules.
        This exposition is utilizing  a complete reducibility criterion
for tensor products of irreducible modules of highest weight found in \cite{M1}. We also modify it
to milder restrictions, as a sufficient condition for the tensor product to be a sum of submodules of highest weight,  relaxing irreducibility of one tensor factor. Let us remind the finding of \cite{M1} first.

Recall that a module of highest weight $Z$ with highest vector $1_Z$ has a unique contravariant form such that
$( 1_Z,1_Z)=1$. The module is irreducible if and only if its contravariant form is non-degenerate.
We denote by $\eps_Z$ a linear map $Z\to \C$ acting by $\eps_Z(z)= (z,1_Z)$ for all $z\in Z$.

Tensor product of highest weight modules  $V\tp Z$ is equipped with a canonical contravariant form that is the product of contravariant forms
on the tensor factors.
Regard $Z$ as a cyclic $U_q(\g_-)$-module generated by the highest vector $1_Z$ and let $J^-\subset U_q(\g_-)$ denote
a $U_q(\h)$-graded finitely generated left ideal  lying in the annihilator of $1_Z$. If $J^-$ exhausts all of the annihilator,
then we have an isomorphism $Z\simeq U_q(\g_-)/J^-$. Similarly we introduce a
left ideal $J^+=\si(J^-)\subset U_q(\g_+)$. It kills the lowest vector in the  opposite module $Z'$ of lowest weight that is
 negative highest weight of $Z$.

Denote by $V^{J^+}\subset V$  the kernel of $J^+$.
It is the annihilator of  the vector space $\omega(J^+)V=\gm^{-1}(J^-)V$ (and {\em vice versa} thanks to finite dimensionality of weight subspaces) with respect to the contravariant form.
If $Z$ is irreducible, then $Z^*\simeq Z'$. In that case, if   $J^+$ is  the  annihilator of the lowest vector in $Z'$, then $V^{J^+}\simeq \Hom_{U_q(\g_+)}(Z^*,V)$.

Now suppose that  $Z$ is irreducible
and $J^-$ is  the {\em entire} annihilator of $1_Z$.
Let $(V\tp Z)^+$ denote the span of extremal vectors in $V\tp Z$ (the subspace of $U_q(\g_+)$-invariants).
 There is a linear isomorphism $\dt_V\colon V^{J^+}\to (V\tp Z)^+$  that
  is the inverse to $\id \tp \eps_Z$ restricted to $(V\tp Z)^+$.
 The pullback of the canonical form from $(V\tp Z)^+$ to
$V^{J^+}$  via $\dt_V$ defines a linear map $\theta\colon V^{J^+}\mapsto V/\omega(J^+)V$ (the space of coinvariants of the right ideal $\omega(J^+)$).
\begin{thm}[\cite{M1}]
\label{com_red_crit}
Let $V$ and $Z$ be irreducible $U_q(\g)$-modules of highest weight.
Then the following assertions are equivalent:
\begin{enumerate}
  \item[i)]  $V\tp Z$ is completely reducible,
  \item[ii)] $\theta$ is bijective,
  \item[iii)] all submodules of  highest weight in $V\tp Z$ are irreducible,
  \item[iv)] $V\tp Z$ is the sum of submodules of highest weight.
\end{enumerate}
\end{thm}

Relaxing the irreducibility  assumption on $Z$  we are looking for a sufficient condition
for $V\tp Z$ to be  a sum of submodules of highest weight. We would like to mimic the above criterion
in a situation when we do not know the annihilator of the highest vector of $Z$ but only a "part" of it.
The new input ingredient that compensates this deficit  of information  is the extremal projector and its relation with the extremal twist
  \cite{M2}.

For modules whose weights $\mu$  are in  $-\Gamma_++\nu$ for some $\nu \in \h^*$ (e.g. modules of highest weight and their
tensor products), we define height of $\mu$ as
the number of simple roots in   $\nu-\mu\in \Gamma_+$. Height of a weight vector is defined as the height of its weight.
If $V$ is equipped with a contravariant form, then extremal vectors of different heights and the modules they generate
are orthogonal to each other. For a module $V$ equipped with height function let $V_k$ denote its submodule generated by vectors of height $\leqslant k$.
It is known that $(V\tp Z)_k$ is  generated by tensors  of height $\leqslant k$ from $V\tp 1_Z$, \cite{M1}, Corollary 5.2.

We still assume that $V$ is irreducible but we do not require the left ideal $J^-\subset U_q(\g_-)$ be the entire annihilator of $1_Z$.
We  define $J^+=\si(J^-)\subset U_q(\g_+)$ as before.

Suppose that $V^{J^+}$ is in the range of $\id\tp \eps_Z$ restricted to $(V\tp Z)^+$ and define a
 $U_q(\h)$-affine (preserving weights up to a constant summand) section $\dt_V\colon V^{J^+}\to (V\tp Z)^+$ of  $\id \tp \eps_Z$.
Consider the pull-back to $V^{J^+}$ of the canonical form via the map $\dt_V$ and
define the extremal twist $\theta\colon V^{J^+}\to V/\omega(J^+)V$ via $\bigl(\theta(v),w\bigr)=\bigl(\dt_V(v),\dt_V(w)\bigr)$ for all $v,w\in V^{J^+}$,
as before. Clearly $\theta$ commutes with the action of $U_q(\h)$.

The map $\dt_V$  preserves   height because  it shifts weights by the highest weight  of $Z$.
\begin{propn}
  Suppose that the $\dt_V$-pullback of the canonical form is non-degenerate on $V^{J^+}$.
  Then $V\tp Z$ is a sum of submodules of highest weights whose highest vectors are from $\dt_V(V^{J^+})$.
\label{TP=sum}
\end{propn}

\begin{proof}
Denote by $V\boxtimes Z$ the sum of submodules generated by extremal vectors from $\dt_V(V^{J^+})$.
Clearly $(V\boxtimes Z)_k\subset (V\tp Z)_k$.
  The assertion will be proved if we demonstrate the reverse inclusion.

   Suppose we proved the required inclusion for $k\geqslant 0$ (it is obviously true for $k=0$).
   Pick up $v\in \omega(J^+)V$ of height $k+1$ and present it as $v=\sum_i \omega(e_i) v_i$, where $e_i\in J^+$ and $v_i\in V$ are some vectors of height $\leqslant k$.
  By Lemma 5.1 from \cite{M1}, $v\tp 1_Z=\sum_{i} v_i\tp \si(e_i)1_Z=0$  modulo $(V\tp Z)_k$, that is
  $v\tp 1_Z \in (V\boxtimes Z)_k$ by the induction assumption.

  Furthermore, if $v'\in V^{J^+}$ of height $k+1$, then there is $v\in V$ of height $k+1$ such that
  $\dt_V(v')=v\tp 1_Z$ modulo $(V\tp Z)_k$
   by  \cite{M1}, Lemma 5.1. The vector $\theta(v')$ is the projection of $v$ along $\omega(J^+)V$ because
   all $\dt_V(w')$ with $\hth(w')=k+1$ are orthogonal to
   extremal vectors of smaller heights and therefore to all $(V\tp Z)_k$ by the induction assumption:
   $$
   \bigl(\theta(v'),w'\bigr)=\bigl(\dt_V(v'),\dt_V(w')\bigr)=\bigl(v\tp 1_Z,\dt_V(w')\bigr)=(v,w').
   $$
   By the hypothesis, the map $\theta\colon V^{J^+}\to V/\omega(J^+)V$ is surjective (and preserves heights because it preserves weights). Then each tensor $v\tp 1_Z$ from $V\tp 1_Z$ of height $k+1$ can be presented
   as $\dt_V(v')$ modulo $(V\tp Z)_k$ plus a tensor from $\omega(J^+)V\tp 1_Z$ of height $k+1$, which is also in $(V\tp Z)_k$
   as already  proved. Therefore the tensor $v\tp 1_Z$ is in $(V\boxtimes Z)_{k+1}$ for all $v$ of height $k+1$,
   as required.
   This implies  $(V\boxtimes Z)_{k+1}\supset (V\tp Z)_{k+1}$. Induction on $k$ completes the proof.
  \end{proof}

We will construct the map  $\dt_V$  with the help of extremal projector provided it can be regularized on an appropriate subspace, cf. \cite{M2}.
Let $\widehat{V^{J^+}}\subset V$ denote the sum of weight subspaces in $V$ whose weights are in $\La(V^{J^+})$.
Let   $\zt$ denote the highest weight of $Z$.
\begin{lemma}
  Suppose that projector $p_\g$ is a regular map $\widehat{V^{J^+}}\tp 1_Z\to (V\tp Z)^+$. Then   $V^{J^+}$ contains the range $p_\g(\zt)\widehat{V^{J^+}}$,
  and the subspace $\widehat{V^{J^+}}\cap \omega(J^+)V$ is in its kernel.
 \label{regularization_proj}
\end{lemma}
\begin{proof}
 It is proved in \cite{M2}, Proposition 3.1, that the operator $p_\g(\zt)$ is
 well defined on $\widehat{V^{J^+}}$ and
 $p_\g(\zt)v= (\id\tp \eps_Z)\bigl(p_\g(v\tp 1_Z)\bigr)$ for all  $v\in \widehat{V^{J^+}}$.
Then for all  $w\in \widehat{V^{J^+}}$ and $e\in J^+$:
$$
\bigl(w,ep_\g(\zt)v\bigr)=\bigl(w\tp 1_Z,(e\tp 1) p_\g(v \tp 1_Z)\bigr)=\bigl(w\tp 1_Z,(1\tp  \gm^{-1}(e)) p_\g(v \tp 1_Z\bigr)=0,
$$
  because $\bigl(1_Z,\gm^{-1}(e) z\bigr)=\bigl(\si(e)1_Z,z\bigl)=0$ for all $z \in Z$. Therefore the range of $p_\g(\zt)$
restricted to $\widehat{V^{J^+}}$ is in $V^{J^+}$.
That $\widehat{V^{J^+}}\cap \omega(J^+)V$ is in $\ker p_\g(\zt)$ follows from $\omega$-invariance of $p_\g(\zt)$, cf. Proposition \ref{shifted-proj-invar}.
\end{proof}
Under the assumption of  Lemma \ref{regularization_proj}, we can think that $p_\g(\zt)$ is defined on entire $V$ taking zero value on $V[\mu]$ if $\mu \not \in \La(V^{J^+})$. Such weight subspaces are in $\omega(J^+)V$, and so defined operator is $\omega$-invariant.
The following result is analogous to Theorem 3.2 from \cite{M2}.
\begin{propn}
  Under the assumptions of Lemma \ref{regularization_proj}, suppose  that
  $p_\g(\zt)$ is surjective  onto $V^{J^+}$.   Then
there is a $U_q(\h)$-affine section $\dt_V$ of $\id \tp \eps_Z$ such that $\theta$ is invertible and
   its inverse $V/\omega(J^+)V\to V^{J^+}$ is the map induced by $p_\g(\zt)$.
\label{range of inverse twist}
\end{propn}
\begin{proof}
Define  $\dt_V$
as a composition $V^{J^+}\to V\to (V\tp Z)^+$,
where the left arrow is any $U_q(\h)$-linear section of the map $p_\g(\zt)\colon \widehat{V^{J^+}}\to V^{J^+}$, and the right
map is $v\mapsto  p_\g(v\tp 1_Z)$ for $v\in \widehat{V^{J^+}}$.
By \cite{M2}, Proposition 3.1,  it is indeed a section of $\id \tp \eps_Z$.

By definition of $\theta$, the matrix element  $\bigl(\dt\circ p_\g(\zt)(v), \dt\circ p_\g(\zt)(w)\bigr)$
equals  $\bigl (\theta \circ p_\g(\zt)(v),p_\g(\zt)(w) \bigr )$ for all $v,w\in V^{J^+}$.
On the other hand, it is equal to
$$
\bigl( p_\g(v\tp 1_Z), p_\g(w\tp 1_Z)\bigr )=\bigl( v\tp 1_Z, \omega(p_\g)\circ p_\g(w\tp 1_Z)\bigr )
=\Bigl( v,(\id \tp \eps_Z)\bigl( p_\g(w\tp 1_Z)\bigr) \Bigr) = \bigl( v, p_\g(\zt)(w) \bigr).
$$
Since the image of $p_\g(\zt)$ is $V^{J^+}$, one arrives at   $\theta\circ p_\g(\zt)=\id$ on $V/\omega(J^+)V$.
\end{proof}
\begin{corollary}
\label{corollary_irred}
  Suppose that assumptions of Proposition \ref{range of inverse twist} are fulfilled. Then
  $V\tp Z$ is a sum of submodules of highest weight. If  $Z$ is irreducible, then  $V\tp Z$ is completely reducible
  and $V^{J^+}=\Hom_{U_q(\g_+)}(Z',V)$.
\end{corollary}
\begin{proof}
  The map $\dt_V$ constructed from the extremal projector and the operator  $\theta$ it defines fulfil conditions of Proposition \ref{TP=sum}, hence the
  first part of the statement. The second assertion holds by virtue of Theorem \ref{com_red_crit}.
\end{proof}
Remark that the shifted extremal projector was considered as a form on coinvariants in the special case of Verma modules in \cite{KN}.
 Corollary \ref{corollary_irred} can be viewed as a generalization of  Proposition 2.3 in \cite{KN}.

We denote by $p^{-1}_\g(\zt)$   an arbitrary $U_q(\h)$-linear section of the map $p_\g(\zt)\colon \widehat{V^{J^+}}\to V^{J^+}$.
Note that, for irreducible  $Z$, extremal vectors in $(V\tp Z)^+$ can be alternatively
constructed  via the extremal projector or via a lift $\Sc\in U_q(\g_+)\tp U_q(\g_-)$ of the inverse invariant pairing $Z\tp Z'\to \C$,
see the next section.
The relation is given by the formula
\be
\Sc(v\tp 1_Z)=p_\g\bigl(p^{-1}_\g(\zt)v\tp 1_Z\bigr)
\label{Shap_proj}
\ee
for each weight vector $v\in V^{J^+}$, \cite{M2}. We will use this relation for construction of star product on
 conjugacy classes in Section \ref{Sec_star_prod}.

It is easy to calculate $\theta_{V,Z}$ if the weights of  $V$ are multiplicity free
Suppose that  $\zt$ is a highest weight of module $Z$. Then
$\theta_{V,Z}$ acts on  $V[\mu]$ as multiplication by a scalar, that is proportional to
$
 \prod_{\al\in \Rm^+}\theta^\al_\mu
$, where
\be
\theta^\al_\mu=\prod_{k=1}^{l_{\mu,\al}}\frac{[(\zt+\rho+\mu,\al^\vee)+k]_{q_\al}}{[(\zt+\rho,\al^\vee)-k]_{q_\al}}.
\label{theta igenvalues}
\ee
Here $l_{\mu,\al}$ stands for the maximal integer $k$ such that $e_{\al}^kV^+[\mu]\not =\{0\}$.

\section{Quasi-classical limit of Shapovalov elements}
\label{Sec_Sing_v_parab}

Recall that a generalized parabolic Verma module $M_{\la,\xi}$, where $\la$ is a base weight for $t\in T$ and
 $\xi\in \La^+_\k$ is a dominant weight for its centralizer subalgebra $\k$, is defined as a quotient of the Verma module
 $\tilde M_{\la+\xi}$.
Consider the extremal vector (defined up to a scalar factor) in $\tilde M_{\la+\xi}$ of weight $\la+\xi-m_\al \al$,
where $m_\al=(\xi,\al^\vee)+1$. Let $\phi_{m_\al \al}\in U_q(\g_-)$ denote its lift under the linear isomorphism
$U_q(\g_-)\to \tilde M_{\la+\xi}$. It is called Shapovalov element corresponding to the positive root $\al$ and a positive integer $m_\al$.
With fixed $\la$ and $\xi$, it  is a rational $U_q(\g_-)$-valued function of $q$, cf. a remark after (\ref{base weight}). Since $\phi_{m_\al \al}$ is defined up to a scalar multiple,
we assume that it is regular in a neighbourhood of $1$ and does not vanish at $q=1$.

A key assumption about the initial point $t\in T$ that facilitates our approach to quantization is that it features certain behaviour in the classical limit  in the
following sense.
\begin{definition}
\label{regular ip}
We call the point $t\in T$ quantizable if
 $\lim_{q\to 1}\phi_{m \al}=f_\al^{m}$ for all $m\in \N$ and all $\al\in \Pi_{\k}$,
where $f_\al\in \k_-$ is the classical root vector of root $-\al$.  
\end{definition}
 Let $\l$ denote the maximal Lie subalgebra in $\k$ that is Levi in $\g$, so that
$\Pi_\l=\Pi_\k\cap \Pi_\g$. It is clear that if $\Pi_\k=\Pi_\l$ then $M_{\la,\xi}$ is just the parabolic
Verma module, and the point $t$ is quantizable. Thus this property is questionable  only when the set $\Pi_{\k/\l}=\Pi_\k\backslash \Pi_\l$ is not
empty.

We conjecture that
all $t$ are quantizable for all simple $G$ and prove  that for non-exceptional $G$ in this section.
We do it by a direct analysis of Shapovalov elements using their explicit construction from the inverse Shapovalov form
(as certain rescaled Shapovalov matrix elements).

\subsection{Inverse Shapovalov form and its matrix elements}
In this section we give an explicit construction of Shapovalov elements relative to generalized parabolic Verma modules,
following \cite{M8,M6}.

For each weight $\mu\in \Gamma_+$ put
\be
\eta_{\mu}=h_\mu+(\mu, \rho)-\frac{1}{2}(\mu,\mu) \in \h\op \C.
\ee
Regard it  as an affine function on $\h^*$ by  the assignment $\eta_\mu\colon \zt \mapsto (\mu,\zt+ \rho)-\frac{1}{2}(\mu,\mu)$, $\zt\in \h^*$.

Let $\{h_i\}_{i=1}^{\rk \g}\in \h$ be an orthonormal basis. The element $q^{\sum_i h_i\tp h_i}$ belongs to a completion of $U_q(\h)\tp U_q(\h)$
in the $\hbar =\ln q$-adic topology. Choose an R-matrix of $U_q(\g)$ such that $\check \Ru=q^{-\sum_i h_i\tp h_i}\Ru\in U_q(\g_+)\tp U_q(\g_-)$
and set $\Cc=\frac{1}{q-q^{-1}}(\check \Ru- 1\tp 1)$. Sending the left tensor leg of $\Cc$ to a representation produces a matrix $C$ with entries in $U_q(\g_-)$ which will be used
for construction of Shapovalov elements.

Note that, in the classical limit, the tensor $\Cc$ tends to  $\sum_{\al\in \Rm^+}e_\al\tp f_\al$, where
$e_\al$, $f_\al$ are classical root vectors normalized to $(e_\al,f_\al)=1$ by the ad-invariant form on $\g$.
This fact will be used in the proof of Proposition \ref{Xc_root} below.

Let $\tilde M_\zt$ be an irreducible  Verma module of highest weight $\zt$
and $\Sc\in U_q(\g_+)\hat \tp U_q(\g_-)$ the lifted  inverse of the invariant pairing
$\tilde M_\zt\tp \tilde M_\zt'\to \C$. Pick up $V\in \Fin_q(\g)$
and denote by $S$ the image of $\Sc$  in $\End(V)\tp U_q(\g_-)$.

 The module $\tilde M_\zt$ becomes reducible at certain $\zt$, which results in poles of $\Sc$ (they may not appear in
 $S$ for a particular $V$).
So we can relax the assumption that $\tilde M_\zt$ is irreducible and work with $S$ independently
regarding its entries as rational trigonometric $U_q(\g_-)$-valued functions of $\zt$.

Every pair of vectors $v,w\in V$ define a matrix element  $(v,S_1w)S_2\in \hat U_q(\g_-)$ (Sweedler notation for $S$) with respect to the contravariant form on $V$.
An explicit expression for the matrix entries $s_{ij}=(v_i,S_1v_j)S_2$ of $S$ in an orthonormal weight basis $\{v_i\}_{i\in I}\subset V$
can be formulated using  the language of Hasse diagrams associated with partially ordered sets.
We introduce such an order on $\{v_i\}_{i\in I}$ (equivalently on $I$) by writing $v_i\succ v_j$ if $\nu_i-\nu_j\in \Gamma_+\backslash\{0\}$,
$l=1,\ldots, k$.
The matrix $S$ is triangular: $s_{ii}=1$ and $ s _{ij}=0$ if $\nu_i\not \succeq \nu_j$. The entry $s_{ij}$
is a rational trigonometric function $\h^*\to U_q(\g_-)$, its value carries weight $\nu_j-\nu_i\in -\Gamma_+$.

Set $\check s_{ab}= -(v_b,S_1v_a) S_2[\eta_{\nu_b-\nu_a}]_qq^{-\eta_{\nu_b-\nu_a}}\in U_q(\b_-)$ for each $v_b\succ v_a$.
An explicit formula for $\check s_{ab}$ in terms of the image $C=\sum_{ij}e_{ij}\tp c_{ij}\in \End(V)\tp U_q(\g_-)$ of the element $\mathcal{C}$  can be extracted from    \cite{M6}:
\be
\check s_{ab}= c_{ba}+\sum_{k\geqslant 1}\sum_{v_b \succ v_k\succ \ldots \succ v_1\succ v_a}
c_{b k}\ldots
c_{1 a}\frac{(-1)^kq^{\eta_{\mu_k}}\ldots q^{\eta_{\mu_1}}}{[\eta_{\mu_k}]_q\ldots [\eta_{\mu_1}]_q},
\label{norm_sing_vec}
\ee
where
$
 \mu_l=\nu_{l}-\nu_{a}\in \Gamma_+$, $l=1,\ldots, k.
$
One can see that every node $v_l$ between $v_b$ and $v_a$ contributes  $-\frac{q^{\eta_{\mu_l}}}{[\eta_{\mu_l}]_q}\in \hat U_q(\h)$
to the products
which we call node factor. These node factors may produce singularities when evaluated at a particular weight.

Let $\bt\in \Pi$ be a composite positive root and $\Pi_\bt\subset \Pi$ be the set of simple roots entering
the expansion of $\bt$ over the basis $\Pi$ with  positive coefficients.
Recall that a simple Lie subalgebra, $\g(\bt)\subset \g$, generated by
$e_\al,f_\al$ with $\al\in \Pi_\bt$ is called  support of $\bt$. Its universal enveloping algebra is quantized as a Hopf subalgebra in $U_q(\g)$.
 \begin{definition}
  Let $V$ be a finite dimensional $U_q(\g)$-module and $v_a,v_b\in V$ a pair of vectors of weights $\nu_a,\nu_b$, respectively.
  We call a triple $(V,v_a,v_b)$ admissible $\bt$-representation if $e_\al v_b=0$ for all $\al\in \Pi_\bt$, $v_a=f_\bt v_b$, and  $(\bt^\vee,\nu_b)=1$.
\end{definition}
In other words, a triple $(V,v_a,v_b)$ is admissible if the vector  $v_b\in V$ is extremal for $U_q\bigl(\g(\bt)\bigr)$ and generates a submodule $\simeq \C^2$  of the  subalgebra  $U_q(\g^\bt)$.
Assuming a triple  $(V,v_a,v_b)$ admissible $\bt$-representation we  put $\phi_\bt(\zt)=\check s_{ba}(\zt)$.

\begin{propn}[\cite{M8}]
\label{Shap_el}
Suppose that $(V,v_a,v_b)$ is an admissible  $\bt$-representation.
For  $\zt\in \h^*$ and $m\in \N$
set $\zt_0=\zt$,
$
\zt_k=\zt_{k-1}+\nu_a$, $k=1,\ldots, m$.
 Then the product
\be
\phi_{m\bt}(\zt)=  \phi_\bt(\zt_{m-1})\>\ldots \> \phi_\bt(\zt_0)\> \in  U_q(\g_-)
\label{factor-root-degree}
\ee
is a Shapovalov element for generic $\zt$ satisfying  $q^{2(\zt+\rho,\bt)}=q^{m(\bt,\bt)}$.
\end{propn}

Next we point out  admissible representations for all composite roots of non-exceptional Lie algebras.
Their simple roots are written below in terms of an orthonormal system $\{\ve_k\}_{k=1}^n$:
$$
\Pi_{\s\l(n)}=\{\ve_1-\ve_2,\> \ldots,\>\ve_{n-1}-\ve_{n}\},
\quad
\Pi_{\s\p(2n)}=\{\ve_1-\ve_2,\>\ldots, \>\ve_{n-1}-\ve_{n},\>2\ve_n\},
$$
$$
\Pi_{\s\o(2n+1)}=\{\ve_1-\ve_2,\>\ldots, \>\ve_{n-1}-\ve_{n},\>\ve_n\},
\quad
\Pi_{\s\o(2n)}=\{\ve_1-\ve_2,\>\ldots, \>\ve_{n-1}-\ve_{n},\>\ve_{n-1}+\ve_{n}\}.
$$
We will enumerate them from left to right.
\begin{propn}
\label{adm_rep}
  For each composite root $\bt\in \Rm^+_\g$   there is an admissible $\bt$-representation.
\end{propn}
\begin{proof}
In all cases except for short roots of $\s\o(2n+1)$ we take for $V$ the natural $\g$-module of minimal dimension. Then

$\g=\s\l(n)$: $v_b=v_{\ve_i}$, $v_a=v_{\ve_j}$ $\bt=\ve_{\ve_i}-\ve_{\ve_j}$, $i<j$.

$\g=\s\p(2n), \s\o(2n),\s\o(2n+1)$: $v_b=v_{\ve_i}$, $v_a=v_{\pm\ve_j}$ for $\bt=(\ve_i\mp \ve_j)$, $i< j$.

$\g=\s\p(2n) $:   $v_b=v_{\ve_i}$, $v_a=v_{-\ve_i}$ for  $\bt=2 \ve_i $.

\noindent
For short roots $\bt =\ve_i$ of $\s\o(2n+1)$ we take for $V$ the fundamental spin module with
  $\nu_b=\frac{1}{2}\sum_{l=1}^n\ve_l$ and $\nu_a=\frac{1}{2}\sum_{l=1}^n\ve_l-\ve_i$,
\end{proof}
In what follows we assume that the admissible triples $(V,v_a,v_b)$ are fixed as in Proposition \ref{adm_rep}.
Observe that in all cases the dimension of weight spaces in $V$ is $1$.
For each   simple root $\al \in \Pi$,
$$
e_\al \phi_{m\bt}(\zt)1_\zt\propto \dt_{\nu_b-\nu_c,\al} (q^{2(\zt+\rho,\bt)}-q^{m(\bt,\bt)}) \psi(\zt)1_\zt,
$$
where  $\psi(\zt)\in U_(\g_-)$  \cite{M8}.
The  vector $\phi_{m\bt}(\zt)1_\zt$ is  extremal if $\zt$ satisfies $q^{2(\zt+\rho,\bt)}=q^{m(\bt,\bt)}$, and
the elements $\phi_{m\bt}(\zt)$ and  $\psi(\zt)\not =0$. Generically these conditions are fulfilled but
we are interested in very special $\zt$ that is a sum of base weight  $\la$ and $\xi\in \La^+_\k$.
We require that $\phi_{m\bt}(\zt)$ is regular in $q$ in a neighbourhood of  $q=1$ for every such $\zt$. Moreover,
the element $\phi_{m\bt}$ should have a proper classical limit $q\to 1$ with fixed $\zt$.

Factorization of Shapovalov elements   reduces the  problem of regularity
to the question about $\phi_{\bt}(\zt)$.
Observe from (\ref{norm_sing_vec}) that the node factors tend to zero  for generic $\zt$, whence $\phi_\bt(\zt_k)$ tends
to the classical root vector $f_\bt$. However   $\phi_\bt(\zt_k)$ may have poles at special $\zt_k$, and  regularized $\phi_\bt(\zt_k)$ may fail to tend to $f_\bt$.
In the next section we demonstrate that admissible triples from Proposition \ref{adm_rep}
guarantee  the proper classical limit of $\phi_\bt(\zt_k)$  for  all $t$.

Once the triple $(V,v_a,v_b)$ has been fixed, the sequence of weights $(\zt_k)_{k=0}^m$ from Proposition \ref{Shap_el} depends only on $\zt=\zt_0$.
Abusing notation we will write $\phi_\bt^m(\zt)=\phi_\bt(\zt_{m-1})\>\ldots \> \phi_\bt(\zt_0)$ or simply
$\phi_\bt^m$ when the weight $\zt$ is clear from the context. This convention will unify notation with the case of  $\bt\in \Pi_\l$, when
the  Shapovalov element $\phi_\bt^m$ is a true power of the Chevalley generator $f_\bt$,
which is of course  independent of $\zt$.

\subsection{Regularity of Shapovalov elements}
In the previous section we presented  a construction of extremal vectors  in Verma modules
from matrix elements of the Shapovalov form. We discussed that it apparently works  for "generic" weight
satisfying a particular Kac-Kazhdan condition. When it comes to a special weight from $\la+\La_\g$,
some node factors may get singular, and properties of  regularized matrix entries are not obvious.
This problem is solved in this section for all initial points and their base weights.

Fix $t\in T$ with the  centralizer  $\k$, a base weight $\la$,  and pick up $\bt \in \Pi_{\k}$.
\begin{definition}
\label{mid_node_root}
We call an admissible $\bt$-representation $(V,v_a,v_b)$   $t$-regular if
for each $v_c\in V$ of weight $\nu_c$ such that
$
v_b\succ v_c\succ v_a
$
and all $\zt\in \la+\La_{\g}$
the $U_q(\g_-)$-valued function  $q\mapsto s_{ca}(\zt,q)$ is regular at  $q=1$ and $s_{ca}(\zt,1)=0$.
\end{definition}
It follows from  (\ref{norm_sing_vec}) that being regular depends only on $t$ and not on a choice of base weight $\la$
because the node factors essentially involve squared $q^{(\la, \mu)}$ with $\mu \in \Gamma$, cf. (\ref{base weight}).

In particular, regularity implies that specialization at $\zt$ makes $s_{ca}$ a well defined rational function of $q$.
Clearly if the root $\bt$ is simple, then its any representation is $t$-regular because there is no node between $v_b$ and $v_a$
to violate the conditions.

We call a  node $v_c$ between $v_b$ and $v_a$  $t$-singular if $\nu_c(t)=\nu_a(t)=\nu_b(t)$, that is,
if $\mu_c(t)=1$ for $\mu_c=\nu_c-\nu_a$. Then the node factor $-\frac{q^{\eta_{\mu_c}}}{[\eta_{\mu_c}]_q}$ in (\ref{norm_sing_vec}) evaluated at $\zt \in \la+\La_\g$ may not
vanish in the classical limit $q\to 1$. As a consequence, the matrix element $s_{ca}$ may not vanish as $q\to 1$.
On the contrary, if all nodes between $v_b$ and $v_a$ are non-singular relative to $t$, then the $\bt$-representation $(V,v_a,v_b)$ is $t$-regular
because all node factors between the end nodes go to zero at $q=1$.

A special case of regular representation of $\bt \in \Pi_\k$ is realized when
both weight differences $\nu_b-\nu_c$ and $\nu_c-\nu_a$ are roots for all $v_c$ between $v_b$ and $v_a$. We call such nodes $v_c$ root splitting.
Then the $\bt$-representation  is $t$-regular
for any $t$ for which  $\bt \in \Pi_{\k}$. Indeed, there is no $t$-singular node $v_c$ between $v_b$ and $v_a$
since otherwise $\bt=(\nu_b-\nu_c)+(\nu_c-\nu_a)$ is a sum of two other
roots from $\Rm^+_\k$, which is impossible.
\begin{propn}
\label{reg=quant}
Suppose that  there is a $t$-regular $\bt$-representation $(V,v_a,v_b)$ for each $\bt\in \Pi_\k$. Then $t$ is quantizable. 
\end{propn}
\begin{proof}
Element $\phi_{m\bt}$ admits a factorization 
\be
\theta_{\bt,m}(\zt)=  \theta_\bt(\zt_{m-1})\>\ldots \> \theta_\bt(\zt_0)\> \in  U_q(\g_-)
\label{factor-root-degree1}
\ee
where $\zt_0=\zt$ and $\zt_k=\zt_{k-1}+\nu_a$, $k=1,\ldots, m$, \cite{M8}. Therefore it is sufficient to consider the case $m=1$. Up to a non-zero scalar multiplier, the summation formula (\ref{norm_sing_vec})  can be rewritten as
\be
\label{ABRR-equiv}
\check s_{ba}= c_{ba}+\sum_{v_b\succ v_c\succ v_a}c_{bc}s_{ca}.
\ee
Only the first term survives in the classical limit, and it goes to the classical root vector $f_\bt$. 
\end{proof}
\subsection{All points of the maximal torus are quantizable}
In this section, we prove that all  points from the maximal torus are quantizable for $G$  of type $A$, $B$, $C$, and $D$.
Specifically, we will show that the admissible $\bt$-representations from
Proposition \ref{adm_rep} are $t$-regular for each  $\bt\in \Pi_{\k/\l}$ and all $t\in T\subset G$.

A node factor in (\ref{norm_sing_vec}) evaluated at weight $\zt=\la+\xi$, $\xi \in \La_\g$ reads
\be
 -\frac{q^{\eta_\mu(\zt)}}{[\eta_\mu(\zt)]_q}\propto \frac{q-q^{-1}}{\mu(t)q^{2(\xi+\eta,\mu)-\frac{1}{2}(\mu,\mu)}-1},
 \label{node factor}
\ee
If the node is not singular,  i.e. $\mu(t)\not =1$, it may
have at most a finite set of poles, as a function of $q$. In the special case of $t\in T_\Qbb$,  there are no poles at all
because $\mu(t)$ is a root of unity while $q$ is not.
 Such  factors tend to zero in the classical limit. If the node is $t$-singular,
then the analysis is more delicate. We will come across it when doing orthogonal $\g$.

 \begin{propn}
All $t\in T$ are quantizable for  $\g=\s\l(n)$ or  $\g=\s\p(2n)$.
\end{propn}
\begin{proof}
For each composite root, the module $V$ in the triple $(V,v_a,v_b)$ from Proposition \ref{adm_rep} is the natural representation of minimal dimension.
The partial ordering in $V$ is total, and the difference between any pair of distinct weights from $\La(V)$ is a root.
So every node between $v_b$ and $v_a$ is root splitting, and no node factor is singular.
Now the proof follows because all  factors  (\ref{node factor}) go to zero
in the classical limit.
\end{proof}
\subsubsection{Orthogonal $\g$}
There are natural $\s\l(n)$-subalgebras in orthogonal $\g$ of rank $n$.
Their composite roots have been treated in the previous section. We will consider only complementary roots below.

First suppose that $\g=\s\o(2n+1)$ and $\bt=\ve_i$ is a short root. The admissible triple is realized in the spin module $V$ with
 $v_b$ of weight $\frac{1}{2} (\ve_1+\ldots+\ve_n)$ and $v_a=v_b-\ve_i$.
 The Hasse sub-diagram
between these points is linear:
 \begin{center}
\begin{picture}(280,30)

\put(35,20){$e_{\al_n}$}
\put(235,20){$e_{\al_i}$}

\put(5,0){$v_{b}$}
\put(255,0){$v_{a}$}

\put(10,15){\circle{3}}
\put(60,15){\circle{3}}
\put(210,15){\circle{3}}
 \put(260,15){\circle{3}}
 \put(55,15){\vector(-1,0){40}}
\put(105,15){\vector(-1,0){40}}
\put(125,15){$\ldots$}
\put(205,15){\vector(-1,0){40}}
\put(255,15){\vector(-1,0){40}}
 \end{picture}
\end{center}
Every node  between $v_b$ and $v_a$ splits $\ve_i$ to the sum of two roots $\ve_i=(\ve_{i}-\ve_k)+\ve_k$ for some
$i<k<n$. Therefore this $\ve_i$-representation  is $t$-regular for all $t$.

Let us turn to the case of long roots. Set $N$ to be the dimension of the minimal fundamental representation
of $\g$.
\begin{lemma}
\label{SO-roots}
If $\bt=\ve_i+\ve_j\in \Pi_{\k}$ and $\al=\ve_i-\ve_j\in \Rm^+_\k$ for $i<j$, then $\ve_i(t)=\ve_j(t)=-(\pm 1)^{N+1}$.
\end{lemma}
\begin{proof}
The inclusion  $\al, \bt \in \Rm^+_\k$ implies $\ve_i(t)=\ve_j(t)=\pm 1$.
Now suppose that $\g=\s\o(2n+1)$ and $\ve_i(t)=\ve_j(t)=1$.
Then $\ve_i,\ve_j\in \Rm^+_\k$ and $\bt=\ve_i+\ve_j$  is not simple in $\Rm^+_\k$
which is a contradiction.
\end{proof}
\noindent
Denote $d_j=q^{\eta_{2\ve_j}}+1=q^{2\eta_{\ve_j}-1}+1$ if $\g=\s\o(2n+1)$  and $d_j=q^{2\eta_{2\ve_j}}-1$ if $\g=\s\o(2n)$.
  If the point $t$ underlying a base weight $\la$ satisfies   $\ve_j(t)=-(\pm 1)^{N+1}$,
then $d_j(\la)$ tends to zero as $q\to 1$.

Now  let  $V$ be the fundamental $U_q(\g)$-module of minimal dimension,
with the set of weights $\La(V)=\{\pm\ve_i\}_{i=1}^n$ in the even and $\La(V)=\{\pm\ve_i\}_{i=1}^n\cup \{0\}$ in the odd cases.

\begin{lemma}
Suppose that $\bt=\ve_i+\ve_j\in \Pi_{\k}$, with  $i<j$.
Then the triple $(V,v_i,v_{-j})$ is a $t$-regular $\bt$-representation.
\end{lemma}
\begin{proof}
The only node between $v_i$ and $v_{-j}$ whose Cartan factor may be  $t$-singular is $v_j$, because all other nodes split $\bt$
into sum of two roots. Therefore $s_{l,-j}$ are regular at almost all $q$ (at all if $t\in T_\Qbb$) and tend to zero as $q\to 1$ for all $l$
such that  $j\succ l\succ -j$.
Let us prove that for $s_{j,-j}$ assuming $\ve_j(t)=-(\pm 1)^{N+1}$,
as in Lemma \ref{SO-roots}. This matrix element has an apparent singularity in the classical limit
because its denominator includes $[\eta_{2\ve_j}]_q$ divisible by $d_j$. We will show that the singularity is removable.
Without loss of generality we will assume $j=1$ (otherwise we should proceed to the quantum subgroup
with roots $\al_j,\ldots, \al_n$).

 Notice that $d_1(\zt)=0$ does not imply the q-version of Kac-Kazhdan condition for any root,
\cite{DCK}.
Therefore we can choose an irreducible Verma module $\tilde M_\zt$ of generic highest weight $\zt$ such that $d_1(\zt)=0$.
The  extremal vector
$$-v_{-1}\tp q^{-\eta_{2\ve_1}}[\eta_{2\ve_1}]_q1_\zt +\ldots +v_1\tp \check s_{1,-1}1_\zt\in V\tp \tilde M_\zt
$$
vanishes at such $\zt$ because the leftmost term does (otherwise its $\tilde M_\zt$-components span a $U_q(\g_+)$-submodule  which
contains an extremal vector in $\tilde M_\zt$ distinct from $1_\zt$). Therefore $\check s_{1,-1}$ is divisible by $d_1$
 and  $\check s_{1,-1}=\psi d_1$, where $\psi$ is regular at $d_1=0$
for almost all $q$ (all if $t$ is of finite order) including $q=1$.
That is also true for $s_{1,-1}=\check s_{1,-1}(q-q^{-1})/(q^{-2\eta_{2\ve_1}}-1)$ at $\zt \in \la+\La_\g$, thanks to Lemma \ref{SO-roots}.
Moreover, $s_{1,-1}=(q-q^{-1})\psi  \> d_1/(q^{-2\eta_{2\ve_1}}-1)$ vanishes in the classical limit $q\to 1$.

Using the presentation
(\ref{ABRR-equiv}) for all nodes $l$ between $i$ and $j$ we conclude that $(v_i,v_{-j})$ is indeed
a $t$-regular representation of $\ve_i+\ve_j$.
The element $\check s_{i,-j}$ is regular at all $q$ if $t$ is
of finite order, because the only possible singularity  in $s_{j,-j}$ is canceled by the factor $d_j$.
\end{proof}
Summarizing the findings of this section, we conclude that
\begin{propn}
All $t\in T$ are quantizable for orthogonal $\g$.
\end{propn}

We conclude this section with a refinement of Proposition \ref{reg=quant}, which will be needed in a study of the whole collection
of modules $M_{\la,\xi}$ with fixed $\la$.
\begin{propn}
\label{Xc_root}
Suppose that  $\bt \in \Pi_{\k}$ and $m\in \N$ 
and pick up a base weight $\la$ for $t$.
Then
\begin{enumerate}
\item
the Shapovalov element
$
\phi_{m\bt}(\zt)=\check s_{ba}(\zt)$ evaluated at $\zt\in \la+\La_{\k}
$ is regular as a function of $q$ at all $q\in \Crt$,
\item
$\phi_{\bt}(\zt)$ does not vanish at all $q\in \Crt$,
\end{enumerate}
\end{propn}
\begin{proof}
For $\g$ of the classical type, each positive root $\bt$ contains a simple root $\al$ with multiplicity 1.
By \cite{M8} Proposition 4.4, one can assume that $\phi_{\bt}\in U_q(\b_-)$ is such that $\phi_{m\bt} =\phi_{\bt}^m$.
Then $\theta_{\bt,m}(\zt)=\prod_{k=0}^{m-1}\theta_{\bt}(\zt-k\bt)$
for all $\zt$. Since the shift by $\bt\in \Pi_\k$ preserves $\La_\k$, it suffices to consider only $\phi_\bt$.
Furthermore, $\phi_\bt$ can be constructed as a matrix element $\check s_{ba}$ via a Hasse diagram of the form
$$
v_b\quad\stackrel{e_{\al}}{\longleftarrow}\quad f_\al v_b\quad \ldots \quad  v_a,
$$
where the suppressed part is independent of $\al$. It can be realized as a sub-diagram
in the Hasse diagram of the quantized adjoint $\g$-module where $v_b$ is of zero weight, hence  weight differences $\nu_b-\nu_c$ are 
positive roots for all $c\prec b$, cf. \cite{M8} Theorem 5.3.

We saw that $t$-singular nodes  in $s_{ca}$ may appear only for orthogonal $\g$ but the corresponding
node factors cancel. Thus the only node (\ref{node factor}) may contribute to singularity of $\phi_\bt(\zt)$
is one whose weight $\mu$ is a positive root, and this root is in 
$\Rm^+_{\g/\k}$ (otherwise the node will split a simple root  $\bt\in \Pi_\k$ to a sum of two roots from $\Rm^+_{\k}$, which is impossible).
Then every denominator in (\ref{node factor}) is proportional to 
$
\mu(t)q^{2(\xi,\mu)-(\mu,\mu)}-1
$,
and  $\mu(t)\in \Spec(t)$. Therefore this factor does not turn zero at $q\in \C\backslash{\sqrt[\Z]{t}}$.
Thus we have proved 1).

To prove 2), observe that we have chosen the Hasse diagram so that  
 $s_{ca}$ in (\ref{ABRR-equiv}) are independent of the genrator $f_\al$ because the only $\al$-arrow in the subdiagram enters the first term 
$c_{bc}$.
Hence the first term is independent of the sum thanks to the PBW theorem. This proves the second statement.
\end{proof}

 \section{Generalized parabolic categories}
\label{Sec_Gen_parab_Verma_mod}

In this section we continue our study   of generalized parabolic modules introduced in Section \ref{Sec-GPVM-def}.

Fix a point $t\in T$ with its centralizer $\k$  and the maximal Levi subalgebra $\l\subset \k$.
Pick up a base weight $\la\in \h^*$, a $\k$-dominant weight $\xi \in \La^+_\k$, and put $\zt=\la+\xi$.
We have a system of Kac-Kazhdan conditions
(\ref{Kac-Kazhdan})
and the set $\{\phi_\al^{m_\al}1_{\zt}\}_{\al \in \Pi_\k}$ of extremal vectors in the Verma module $\tilde M_{\zt}$.
Denote by $\hat M_{\la,\xi}$ the quotient of $\tilde M_{\zt}$ by the sum of submodules generated
by $\{\phi_\al^{m_\al}1_{\zt}\}_{\al \in \Pi_\l}$. There is a sequence of epimorphisms
$$
\tilde M_{\zt}\to \hat M_{\la,\xi}\to M_{\la,\xi}.
$$
The parabolic Verma module $\hat M_{\la,\xi}$ is locally finite over $U_q(\l)$, \cite{M4}. Therefore
$M_{\la,\xi}$ is also locally finite when restricted to $U_q(\l)$. This fact is of importance for our
further study.

Pick up a composite root $\bt\in \Pi_{\k/\l}$.
It follows from factorization (\ref{factor-root-degree}) and Proposition \ref{Xc_root}
that the Shapovalov elements $\phi^{m_\bt}_\bt$ are  regular once $q\in \Crt$. Since the leading term in $\phi_\bt(\zt_k)$ (the first summand in (\ref{norm_sing_vec})) is the only one that contains a generator $f_\bt$ of the PBW-basis, it
is independent of the other terms. Thus we conclude that $\phi^{m_\bt}_\bt 1_{\zt}$ does not vanish
in $\hat M_{\la,\xi}$ and is an extremal vector, provided $q\in \Crt$.

As a $U_q(\g_+)$-module, $M_{\la,\xi}$ is isomorphic to $U_q(\g_-)/J^-_\xi$, where $J^-_\xi$ is the left ideal annihilating the highest vector in $M_{\la,\xi}$.
It is generated by $\phi_\al^{m_\al}$ with all $\al \in \Pi_\k$.
We will also consider the opposite module $M'_{\la,\xi}$ of lowest  weight $-\la-\xi$. It is isomorphic to the $U_q(\g_+)$-module $U_q(\g_+)/J^+_\xi$,
where $J_\xi^+=\si(J^-_\xi)$.

By $V^+_\xi$ we denote the kernel $V^{J^+_\xi}$ of $J^+_\xi$.
In the classical limit, weight vectors in $V^+_\xi$ are  in bijection with irreducible submodules in the $\k$-module $V\tp X_\xi$,
so $\La(V^+_\xi)+\xi$ is in $\La^+_{\k}$.
We denote by $^+\!V_\xi$ any $U_q(\h)$-invariant subspace in $V$ that is transversal to $\omega(\J^+_\xi)V$, thus we can
identify it with $V/\omega(\J^+_\xi)V$, the  dual to $V^+_\xi$ with respect to the contravariant form on $V$.
One can prove that in the classical limit
the form is non-degenerate on $V^+_\xi$, so one can choose $^+\!V_\xi=V^+_\xi$
for almost all $q$. Then the external twist $\theta$ responsible for  complete reducibility of $V\tp M_{\la,\xi}$
becomes an operator from $\End(V^+_\xi)$.

Much of our further analysis relies on elementary technical facts about base weights
 which we arrange as a separate proposition for further convenience.
\begin{lemma}
Let $t\in T$, $\k$ be its centralizer, and $\la$  a $t$-base weight. Then
\begin{enumerate}
\item
 For each $\al \in \Rm^+_{\g/\k}$ and all $c\in \Qbb$ the number $[(\la+\rho,\al^\vee)+c]_{q_{\al}}$ is not zero
at all $q\in \Crt$.
Its reciprocal tends to zero as $q\to 1$.
\item
For each $\al \in \Rm^+_{\k}$ and any $c\in \Qbb$ the function $q\mapsto [(\la+\rho,\al^\vee)+c]_{q_{\al}}$ is either identically zero or does not vanish at all $q\in \Crte$.
\end{enumerate}
\label{q-num}
\end{lemma}
\begin{proof}
By definition of base weight (\ref{def_base weight}) we write
$$
[(\la+\rho,\al^\vee)+c]_{q_{\al}}
=q^{-(\la+\rho,\al^\vee)-c}_\al\times\frac{\al(t) q^{2(\kappa,\al^\vee)+2c}_\al-1}{q_\al-q^{-1}_\al}.
$$
If  $\al \in \Rm^+_{\g/\k}$ then 1) holds true because  $\al(t)\not=1$.
 Finally, the inverse fraction goes to zero as $q_\al-q^{-1}_\al$ does. This proves the first assertion.

Now suppose that $\al \in \Rm^+_{\k}$ and therefore $\al(t)=1$.  Then $\al(t) q^{2(\kappa,\al^\vee)+2c}_\al=1$
if and only if $c=-(\kappa,\al^\vee)$ as $q$ is not a root of unity.  This proves the second assertion.
\end{proof}
 \begin{propn}
Let $Z$ be a module of highest weight $\zt\in \la + \La$. Suppose that $Z$ is locally finite over $U_q(\l)$.
Then for each $V\in \Fin_q(\g)$ and $\eta\in \La^+_\k$,
the map  $p_\g\colon (V\tp Z)[\la+\mu]  \to (V\tp Z)^+$ is well defined
for   $q\in \Crt$.
\label{projector_is_reg}
\end{propn}
\begin{proof}
 Without loss of generality, we can assume that the base weight $\la$ defines the trivial representation of
   $U_q(\l)$, that is $(\la,\al)=0$ for all  $\al \in \Rm_\l$.

If $\al\in \Rm^+_{\g/\k}$, then all Cartan denominators in $p_\al$ entering factorization (\ref{factorization})
do not turn zero at $q\in \Crt$, and $p_\al$ is regular on $(V\tp Z)[\la+\eta] $. In particular,
for simple $\al$, the operator $p_\al$ sends $(V\tp Z)[\la+\eta] $ to $\ker e_\al$  because $V\tp Z$ is locally nilpotent over $U_q(\g^{\al}_+)$, cf. Lemma \ref{cut_proj}.

Now suppose that $\al \in \Rm^+_{\k}$.
The denominator in (\ref{translationed_proj}) returns
 $[(\kappa+\eta,\al^\vee)+i]_{q_\al}$ on the
weight space $(V\tp Z)[\la+\eta]$. It never vanishes because $(\kappa+\eta,\al^\vee)$ and $i$ are positive integers.
Since $Z$ is locally finite over $U_q(\l)$, the tensor product $V\tp Z$ is locally finite too. Then  $\la+ \eta$ is $\l$-dominant,
 and $p_\al$ is well defined on $(V\tp Z)[\la+\eta] $. Moreover, $(V\tp Z)[\la+\eta] $ is mapped to $\ker e_\al$ by  $p_\al$ if $\al \in \Pi_\l$,
 see \cite{M2}, Proposition 3.6.

Thus the map $p_\g\colon (V\tp Z)[\la+\eta] \to V\tp Z$ is well defined and independent of a normal ordering on $\Rm^+_\g$.
 We are left to show that it ranges in $(V\tp Z)^+$. Fix $\al\in \Pi_\g$ and  order positive roots with  $\al$ in the left-most position.
Then $p_\g$-image of $(V\tp Z)[\la+\eta] $ is killed by $e_\al$ as argued and hence by all $e_\al$ with  $\al \in \Pi_\g$ since  $p_\g$ is independent of normal ordering.
\end{proof}
Remark that the operator $p_\g$ can be factorized as $p_{\g/\l}\times p_{\l}$ as there is a normal ordering with roots from $\Rm^+_\l$ on the right.
The right factor is the extremal projector to the subspace of $U_q(\l_+)$-invariants in $V\tp Z$.

\subsection{Base module}
We start our analysis with the base module $M_\la$ and prove that it is irreducible for
almost all $q$.
Our interest in $M_\la$  is motivated by an idea to represent quantized polynomial ring $\C[O]$ as a subalgebra of linear operators on $M_\la$.
A neighborhood of the initial point   $t\in O$ can be parameterized by the tangent space $\g/\k\simeq \g_-/\k_-\op \g_+/\g_+$, which facilitates an embedding of $\C[O]$
in $\C[\g/\k]\simeq
\C[\g_+/\k_+]\tp \C[\g_-/\k_-]
$. This tensor product looks like a matrix algebra if  there is a suitable pairing between factors making  them  dual vector spaces.
This observation suggests to seek for a deformation of $\C[\g_-/\k_-]$ as a module of highest weight of the same functional dimension
while
a deformation of  $\C[\g_+/\k_+]$ as its opposite module of lowest weight. The required duality will be secured if the modules are irreducible.

It is easy to meet the requirement on the size of $M_\la$ if $\k$ is Levi thanks to  a PBW basis in $U_q(\g_-)$
with all monomials from $U_q(\l_-)$  on the right. The complementary monomials  deliver
a basis on the scalar parabolic Verma module $M_\la$  of highest weight $\la$.
Finding a basis in  $M_\la$ is   challenging for non-Levi $\k$, so we have to resort to deformation arguments.
We regard $M_\la$ as a module over $U_q(\g_-)$ and extend the ring of scalars to $\C_1(q)$. Then
$M_\la$ is a deformation of a classical $U_q(\g_-)$-module, and the deformation respects the weight spaces, up to
the shift by $\la$.
\begin{lemma}
\label{base_char}
The character of the $\C_1(q)$-extension of $M_\la$ equals $\Char(M_\la)=\prod_{\al \in \Rm^+_{\g/\k}}(1-e^{-\al})^{-1}\times e^\la$.
\end{lemma}
\begin{proof}
The  Lusztig  generators of the PBW basis are deformations of  classical root vectors. We redefine $f_\al$ for $\al \in \Rm^+_{\k/\l}$
as follows.
The classical root vector $f_\al$ is a composition of commutators among $f_\bt$, $\bt \in \Pi_{\k/\l}$ and
$f_\nu$, $\mu \in \Pi_\l$. Define its quantum counterpart by replacing $f_\bt$ with $\phi_\bt(\la)$,
simple root vectors from $\l_-$ with Chevalley generators of $U_q(\l_-)$, and
classical commutators with $q$-commutators.
By construction, they are deformations of classical $f_\al$ for all $\al \in \Rm^+_\k$.

Order roots so that $\Rm^+_\k$ is on the right of $\Rm^+_{\g/\k}$ and consider the PBW system of ordered monomials
in the negative root vectors modified as above. The cardinality of this system in every weight space equals its dimension, thanks to the
presence of a standard PBW basis in  $U_q(\g_-)$. In the classical limit, this system delivers a PBW basis in $U(\g_-)$,
therefore it is a basis in every weight subspace of $U_q(\g_-)$, for almost all $q$.
Then the monomials in the root vectors with roots from $\Rm^+_{\g/\k}$ deliver a basis in every
weight subspace of $M_\la$, for almost all $q$.
\end{proof}

The rest of the section is devoted to the question of irreducibility of $M_\la$. Suppose that
$G$ is a  connected  simply connected group with Lie algebra $\g$. Let  $K$ be  its closed subgroup
with the Lie algebra $\k$. It is known that $K$ is connected, \cite{Hum}.
Let $\Xi$ denote the set of isomorphism classes of $\g$-modules $V\in \Xi$ with non-zero space of $\k$-invariants.
Such modules are those appearing in
$\C[G/K]$ thanks to Peter-Weyl decomposition, \cite{GW}.   Since $\k$ is reductive, every module enters $\Xi$ along with its dual.
We will use the same notation for  classes of $U_q(\g)$-modules whose quasi-classical counterparts are in $\Xi$.

Denote by $N^\pm $ the quotients $U(\g_\pm)/U(\g_\pm)\k_\pm$.
It follows from  Lemma \ref{base_char} that each subspace in $M_\la$ of weight $\mu+\la$  has dimension of  $N^-[\mu]$
for almost all $q\in\Crte$.
Every $v\in V^\k$ defines a $U(\g_+)$-invariant map $\hat \varphi_v\colon U(\g_+)\to V$, $x\mapsto x\tr v$, that factors through
a map $\varphi_v\colon N^+\to V$.

\begin{lemma}
\label{class_ind_0}
  The intersection of $\ker \varphi_v$ over all $v\in V^\k$ and all $V\in \Xi$ is zero.
\end{lemma}
\begin{proof}
Pick up a module $V\in \Xi$ and a vector $v\in V^K$ such that $K$ is the isotropy subgroup of $v$. Such a pair $(V,v)$ does exist by  \cite{GW}, Theorem 11.1.13.
The coset space $G/K$ is embedded in $V$ via the map
$g\mapsto g v$,  hence $\varphi_v$ yields an embedding $\g_+/\k_+\to V$.
Fix an order on positive roots and consider a PBW monomial $\prod_{\al \in \Rm^+_{\g/\k}}e_\al^{m_\al}$ of degree $m$.
Apply it to the $K$-invariant tensor $v^{\tp m}$ in the symmetrized tensor power $\mathrm{Sym}( V^{\tp m})$.
The result is a symmetrized tensor  $\propto \mathrm{Sym}\bigl(\tp_{\al \in \Rm^+_{\g/\k}}(e_\al v)^{\tp m_\al}\bigr)$ plus  terms
 $\propto \mathrm{Sym}\bigl(v^{\tp k}\tp (\ldots)\bigr)$ with $k>0$. The first term is independent of the remainder, and furthermore,
the images of the PBW monomials of degree $m$ are independent in $V^{\tp m}$.\end{proof}
It follows that for every weight $\mu\in \La(N^+)$ there is $V\in \Xi$
and a $\g_+$-invariant map  $N^+\to V$ generated by $v\in V^\k=V^{\k_+}[0]$ that is injective on $N^+[\mu]$.
We will mimic this situation in the quantum setting.

Set $Z=M_\la$.
Suppose that $u\in V\tp Z$  is an extremal vector such that $v=(\id \tp \eps_Z)(u)\in V$ is not zero.
Define a linear map $\psi_{v}\colon Z\to V$ as $\psi_{v}(z)=u^{1}( u^2,z)$ (in Sweedler notation), for all $z\in Z$.
It factors through  a composition
$Z\to{}^*\!Z\to V$, where the first arrow is the contravariant form
regarded as a linear map from $Z$ to its restricted ($U_q(\h)$-locally finite) right dual ${}^*\!Z$.
\begin{lemma}
\label{part_iso}
  For every element $f\in U_q(\g_-)$ of weight $-\al$, the mapping $\psi_v$ acts by the assignment $\psi_{v}(f1_Z)=q^{-(\la+\mu,\al)}\si(f)v$,
  where $\mu$ is the weight of $v$.
\end{lemma}
\begin{proof}
  It is sufficient to prove the equality for $f$ a monomial in Chevalley generators.
For simple $\al \in \Pi$ one has
$$
\bigl(1\tp \omega(f_\al )\bigr)u=-(1\tp q^{-h_\al}e_\al )u=-\bigl(\gm(e_\al)\tp q^{-h_\al} \bigr)u=(e_\al q^{-h_\al}\tp q^{-h_\al} )u
=\bigl(\si(f_\al) \tp 1\bigr)q^{-h_\al}u.
$$
This implies $\bigl(1\tp \omega(f)\bigr)u=q^{-(\la+\mu,\al)}\bigl(\si(f)\tp 1\bigr)u$ for all $\al$ and all monomial $f$,
by induction on degree of $f$.
Now the proof is immediate as $\psi_v(f1_Z)= (\id \tp \eps_Z) \Bigl(\bigl(1\tp \omega(f_\al )\bigr)(u)\Bigr)$.
\end{proof}
In other words, $\psi_v$ is a homomorphism of $U_q(\g_-)$-modules. Remark that the vector $v$ is not an arbitrary element
from $\Hom_{U_q(\g_-)}(Z,V)$, but one originating from an extremal vector in $V\tp Z$. This fact is crucial because only
then the homomorphism  $Z\to V$, $f1_Z\mapsto q^{-(\la+\mu,\al)}\si(f)v$, factors through a homomorphism ${}^*\!Z\to V$.  Irreducibility of $Z$ will be proved if the map $Z\to {}^*\!Z$
is shown to have
no kernel and therefore be an isomorphism. It is sufficient to check it only for those weight spaces where
extremal vectors in $Z$ may appear. That is, on the orbit of highest weight $\zt$ of $Z$ under the affine action
$\zt\mapsto w.\zt=w(\zt+\rho)-\rho$ of the Weyl group $\mathrm{W}\ni w$.

The following theorem is one of our main results. We rely in its proof  on the special case of Proposition \ref{extremal_dim} for $\xi=0$.
\begin{thm}
  The base module $M_\la$ is irreducible for almost all $q\in \Crt$.
  \label{base_module_irred}
\end{thm}
\begin{proof}
  It is sufficient to prove that $M_\la$ has no extremal vectors or, equivalently, the contravariant form of $M_\la$ is non-degenerate on $M_\la[\mu]$
  with
  $\mu\in \mathrm{W}.\la$. Let $N^+_q\simeq M'_\la$ denote the quotient of $U_q(\g_+)$ by the left ideal $\J^+$ annihilating the lowest vector
   in module $M'_\la$.
By Lemma \ref{base_char},
$$
\dim N^+_q [\la-\mu]= \dim N^+[\la-\mu]= \dim M'_\la[-\mu]=\dim M_\la[\mu]
$$
for  almost all $q$.
 Pick up a module $V\in \Xi$ such that the weight space $N^+[\la-\mu]$
  is embedded in $V$  in the classical limit via some $\varphi_{v_0}$.
   Proposition \ref{extremal_dim} for $\xi=0$ below guarantees that,
   for almost all $q$, there is a non-zero vector $v= p_\g(\la)v'$, $v'\in V[0]$, deforming $v_0$, such that $u = p_\g(v'\tp 1_\la)$
   is extremal and $v=(\id \tp \eps_{M_\la})(u)$.
      Then $N^+_q[\la-\mu]$ is embedded in $V$ and the map $\psi_v\colon M_\la[\mu]\to N^+_q[\la-\mu]v$
      is an isomorphism for almost all $q$. Since $\mu$ runs over a finite set $\mathrm{W}.\la$, the assertion follows.
\end{proof}
 \subsection{Generalized parabolic Verma modules}
In this section we study the $\Fin_q(\g)$-module category generated by the base module.
As before, $t\in T$ and $\k\subset \g$ is the centralizer of $t$.
Pick up a base weight $\la$ for $t$, a $\k$-dominant weight $\xi\in \La^+_\k$ and set  $\zt=\la+\xi$. Denote
$m_\al=(\xi,\al^\vee)+1\in \N$, for all $\al\in \Pi_\k$.

Let $V$ be a $\g$-module considered as a $\k$ module by restriction.
For a irreducible finite dimensional $\k$-module $X$ set $V^+_X$ to be  $\Hom_{\k_+}(X^*,V)$.

\begin{propn}
 \label{extremal_dim}
Suppose that all weight of  $V\in \Fin_q(\g)$ are multiplicity free. Then for almost all $q\in \Crt$, 
the $U_q(\h)$-modules  $V^+_\xi$ are flat deformations of  $V_X^+$
 and  the operator $p_\g(\zt)\colon \widehat{V^+_\xi}\to {V^+_\xi}$ is surjective.
\end{propn}
\begin{proof}
Denote by $W=\sum_{\mu\in \La(V^+_X)}V[\mu]\subset V$ the sum of all weight spaces whose weights belong to $\La(V^+_X)$.
Thanks to Proposition  \ref{projector_is_reg} the operator $p_\al(\zt)$ is well defined in  $W$ and ranges in $V^+_\xi$
due to Lemma  \ref{regularization_proj}.

Consider the eigenvalues of the foot factors $p_\al(\zt)$, given by (\ref{theta igenvalues}),
on the subspace of weight $\mu\in \La(W)$.
If $\al \in \Rm^+_{\g/\k}$, then neither enumerators  
$\propto [(\la+\xi+\mu+\rho,\al^\vee)+k]_{q_\al}$,
nor denominators $\propto [(\la+\xi+\rho,\al^\vee)+k]_{q_\al}$ turn zero at $q\in \Crt$
By Lemma \ref{q-num}, 1).
We are left to work out the case of $\al\in \Rm^+_\k$.

In the classical limit, all factors  $p_\al(\zt)$ with  $\al \in \Rm^+_{\g/\k}$ turn to identical operator on $V^+_X$.
Then only the factors with $\al\in \Rm^+_\k$ are left, and the  operator $p_\g(\zt)$
turns to  $p_\k(\xi)$ (a normal order on  $\Rm^+_\g$ induces a normal order on $\Rm^+_\k\subset \Rm^+_\g$). But this is the inverse extremal twist $\Theta_{V,X}$ 
for the finite dimensional $\k$-module $X$, which is invertible because of complete reducibility of finite dimensional 
$\k$-modules. Thus all factors  (\ref{theta igenvalues})
are regular in $q$ and do not vanish in some neighbourhood of $q=1$, and thence for all $q\in \Crt$,
by Lemma \ref{q-num}, 2).

We are left to prove that 
$W=V^+_\xi$ except  maybe for a finite subset of  $q\in \Crt$ independent of $\xi$.
Since kernels do not increase in deformation, $ V^+_\xi\subset W$ for almost all $q\in \Crt$. On the other hand, as we proved,
$V^+_\xi\supset p_\g(\zt)W=W$ for all $q\in \Crt$, whence
$V^+_\xi\supset W$ for all and   $V^+_\xi=W$ for almost all   $q\in \Crt$.
As the set of weights in $V$ is finite,  $W$ can be smaller than $V$ only for a finite number of $\xi$
(Shapovalov elements of high degree annihilate all  $V$).
Hence the finite set of  $q\in \Crt$, where equality 
$W=V^+_\xi$ is violated can be chosen independent of $\xi$.
\end{proof}

Recall that the category $\O_q$  consists of finitely generated $U_q(\g)$-modules that are  $U_q(\h)$-diagonalizable and   $U_q(\g_-)$-locally finite.
Tensor product with finite dimensional modules preserves $\O_q$ and makes it a $\Fin_q(\g)$-module category.
Our goal is to study a $\Fin_q(\g)$-module subcategory $\O_q(t)\subset \O_q$ associated with a point $t\in O\cap T$.
Objects of $\O_q(t)$ will be interpreted as "representations" of vector bundles over
the quantized $\C[O]$ provided $\O_q(t)$ is semi-simple, which is the main question to answer.
\begin{definition}
  A full subcategory $\O_q(t)$ of the category $\O_q$ whose objects are  $U_q(\g)$-submodules in  $V\tp M_\la$,
  where $V\in \Fin_q(\g)$, is called generalized parabolic category of the point $t$.
\end{definition}
Category $\O_q(t)$ is obviously additive and stable under tensor product with modules from $\Fin_q(\g)$ by construction.
Although $\O_q(t)$ apparently  depends on $\la$, a different choice of $\la$ results in an isomorphic category,
so we suppress the base weight from notation.

Note that for different points $t\in T\cap O$ the categories $\O_q(t)$ are different although they will be shown equivalent for  almost all $q\in \Crt$.

Fix $V$ to be the fundamental module of minimal dimension for special linear and symplectic $\g$.
For orthogonal $\g$ let $V$ be a fundamental spin module. In all cases, weight subspaces in $V$ have dimension 1.
The simply connected group $G$ with Lie algebra $\g$ is faithfully represented
 in $\End(V)$, and all equivariant vector bundles on $O$ are generated by the  vector sub-bundles appearing in
 $O\times V$.

Let  $L(\mu)$ denote the irreducible $U_q(\g)$-module of highest weight $\mu\in \h^*$.
\begin{thm}
\label{semisimplicity}
For each $t\in T$ and almost all $q\in \Crt$ the following holds true:
\begin{enumerate}
\item
The category  $\O_q(t)$ is semi-simple.
\item
Simple objects in $\O_q(t)$ are exactly   $L(\la+\xi)$  with $\xi\in \La^+_{\k}\cap \La_\g$.
\end{enumerate} 
\end{thm}
\begin{proof}
  Since every finite-dimensional $U_q(\g)$-module is a submodule in a tensor power of $V$, we will prove 1) if we do it
  for  all $V^{\tp m}\tp M_\la$ using induction on  $m\in \Z_+$ (for $m=0$ this is Theorem \ref{base_module_irred}).
  Assuming  that  $V^{\tp m}\tp M_\la$ is completely reducible and its all simple submodules
  are $L(\la+\xi)$ such that $X_\xi\subset V^{\tp m}$ we will prove 1) for each $V\tp L(\zt)$, $\zt=\la+\xi$, as the induction transition.

  Suppose that we did it for some $m\geqslant 0$.
  Let  $V^{\tp m}\tp M_\la=\op_i L(\zt_i)$ with  $\zt_i=\la+\xi_i$ be an irreducible decomposition and set $\zt$ to one of $\zt_i$. Then complete reducibility of $V\tp L(\zt)$ is a consequence of Proposition \ref{extremal_dim}, 1) and Corollary \ref{corollary_irred}.

  It is clear from Proposition \ref{extremal_dim},1) that every module $L(\la+\xi)$ with $\xi\in \La^+_{\k}\cap \La_\g$ appears in $V^{\tp m}\tp M$ for some $m$ because $X_\xi$ appears in some
  $V^{\tp m}$.  This proves 2)..
\end{proof}
It follows that  the category $\Q_q(t)$ is semi-simple at those $q\in \Crt$ where 
the dimensionality of the generalized extremal spaces  
equals classical and  the base module is irreducible.
We denote by $\Omega_t\subset \Crt$ the set of such  $q$. It is open with respect to the topology
 induced by the Zariski topology on  $\C$.

\begin{propn}
\label{pseudo-parabolic-irreducible}
For all $q\in \Omega_t$ and all $\xi\in \La^+_{\k}\cap \La_\g$,
  $\Char\bigl(L(\la+\xi)\bigr)=\Char(X_\xi)\Char(M_\la)$.
\end{propn}
 \begin{proof}

 Consider $M_{\la,\xi}$ as a $U_q(\g_-)$-module. It the classical limit, it  goes to the quotient of $U(\g_-)$ by the left ideal generated by the annihilator
of the highest vector in $X_\xi$.
Therefore
$$
\Char\bigl(L(\la+\xi)\bigr)\leqslant \Char (M_{\la, \xi})\leqslant \Char(X_\xi) \Char(\g_-/\k_-)e^{\la}=\Char(X_\xi)\Char(M_\la)
$$
over $\C_1(q)$
meaning inequality of each weight space dimension for $q$ from a punctured neighbourhood of $1$ (which depends on the weight space in general).

Suppose that the statement is true for all  $L(\la+\xi)\subset V^m\tp M_\la$ with some $m\geqslant 0$ (that is obviously so  for $m=0$).
The direct sum decomposition $V\tp L(\la+\xi)=\sum_i L(\la+\xi_i)$
implies
\be
\Char(V) \Char\bigl(L(\la+\xi)\bigr)&=&\sum_i\Char\bigl(L(\la+\xi_i)\bigr)
    \leqslant\sum_i\Char (M_{\la,\xi_i}) \leqslant
   \nn\\
    &\leqslant& \sum_i\Char(X_{\xi_i})\Char(M_\la)=\Char(V)\Char(X_\xi)\Char(M_\la)
\label{car-eq}
\ee
over $\C_1(q)$ because $V\tp X_\xi=\sum_i X_{\xi_i}$. Therefore the inequalities are all equalities.
Furthermore, for each $i$ and each weight $\mu$ an equality
\be
\dim L(\zt_i)[\mu] =\dim M_{\la,\xi_i}[\mu]=\dim (X_{\xi_i}\tp M_\la)[\mu]
\label{eq_char}
\ee
holds for almost all $q\in \Omega_t$.
But then $\Char \bigl(L(\zt_i)\bigr)\geqslant \Char(X_{\xi_i})\Char(M_\la)$ for all $q\in \Omega_t$ as $L(\zt_i) $ is
a quotient of a Verma module, which is flat at all $q\in \Crte\supset \Crt$.
If the inequality is strict for some $i$, then
\be
\Char(V) \Char\bigl(L(\la+\xi)\bigr)&=&\sum_i\Char \bigl(L(\la+\xi_i)\bigr)
 >\sum_i\Char\bigl(X_{\xi_i}\bigr)\Char(M_\la).
\nn
\ee
But this is impossible because the left- and right-most terms are equal to $\Char(V)\Char(X_\xi)\Char(M_\la)$.
Therefore $\Char \bigl(L(\zt_i)\bigr) = \Char(X_{\xi_i})\Char(M_\la)$ in  $\Omega_t$.
This is true  for
all $\xi_i$ such that
$X_{\xi_i}\subset V^{\tp(m+1)}$ and therefore for all $\xi \in \La_{\k}^+\cap \La_\g$, by induction.
This completes the proof.
 \end{proof}
Now we describe the irreducible submodules in $\O_q(t)$ and show that they are essentially generalized parabolic Verma
modules.

\begin{corollary}
\label{pseudo-parabolic-irreducible1}
  For each $q\in \Crt$ and every  $\xi\in \La^+_{\k}\cap \La_\g$ the module $M_{\la,\xi}$ is irreducible, and
  $\Char(M_{\la,\xi})=\Char(X_\xi)\Char(M_\la)$.
 \end{corollary}
 \begin{proof}
The module  $L(\la+\xi)$ is a quotient of  $M_{\la,\xi}$, so it suffices to check that their infinitesimal 
characters are equal. We have checked equality of weight spaces in  (\ref{eq_char}) for each weight 
and almost all $q\in \Omega_t$ and then extended it for all $q\in \Omega_t$. 
This proves irreducibility of 
$M_{\la,\xi}$,
the the character formula follows from Proposition  \ref{pseudo-parabolic-irreducible}.
 \end{proof}

\section{Quantization of associated vector bundles}
\label{Sec_star_prod}
A construction of equivariant star product on homogeneous spaces with Levi isotropy subgroups was discovered  about twenty years ago
\cite{AL,DM,EE,EEM}.
It was employing dynamical twist, or equivalently, the inverse invariant pairing between parabolic base module $M_\la$ and its opposite $M_\la'$.
A  lift of the form to $U_q(\g_+)\tp U_q(\g_-)$  delivers a  quasi-Hopf algebra twist of $U_q(\g)$, \cite{D2}.
A coherent twist of its dual algebra of functions on the quantum group turns out to be associative on the subspace of $U_q(\k)$-invariants.

Algebraically this construction works in a more general setting than parabolic Verma modules, \cite{KST},
however there is problem of the size of "$\k$-invariants" in the absence of the quantum subgroup $U_q(\k)\subset U_q(\g)$.
We worked it out for even quantum spheres in \cite{M3} through harmonic analysis on the  quantum Euclidean plane.
In this section we extend that result for all conjugacy classes $O(t)$, $t\in T$. Moreover,
we put it in a more general context of quantum vector bundles, in development of the approach of \cite{DM}.
Furthermore, we obtain an explicit presentation of the star product by expressing it through the extremal projectors, similarly to
\cite{M2}.
\subsection{Equivariant star product}
Let $\Tc$ be the Hopf algebra quantization of $\C[G]$ along the Drinfeld-Sklyanin Poisson bracket, \cite{FRT}.
It is known to be a local star product (the multiplication is delivered by a bi-differential operator)  \cite{T,EK}.
The quantum group  $U_q(\g)$ enjoys a two-sided action on $\Tc$ by left and right translations.
The Peter-Weyl decomposition splits $\Tc$  in the direct sum $\op_{[V]} V^*\tp V$  over all equivalence classes of irreducible finite-dimensional $U_q(\g)$-modules.
The structure of a $U_q(\g)$-bimodule descends from a realization of $\Tc$ by matrix elements of representations: $x\tr (v^*\tp v)= (v^*\tp xv)$ and $( v^*\tp v) \tl x= (v^*x\tp v)$, where we assume the natural right action on the dual space by transposition.
In terms of Hopf pairing between $\Tc$ and $U_q(\g)$ they can be  written as
$$
x\tr a=a^{(1)}(a^{(2)},x), \quad a\tl x=(a^{(1)},x)a^{(2)}
$$
for all  $a\in \Tc$ and $x\in U_q(\g)$. This implies that $\Tc$ is a module algebra with respect to the left action $\tr$.
The opposite multiplication on $\Tc$ is equivariant with respect to the left action $x\diamond a= a\tl \gm(x)$.

With every finite-dimensional irreducible $\k$-module $X\in \Fin(\k)$ one associates an equivariant vector bundle over the base
$O$ with fiber $X$.  With a realization of $\C[O]$ as the subalgebra of $\k$-invariants  in $\C[G]$ under the action $\tr$ (the classical limit of),
the $\C[O]$ -module of global sections can be realized as
$\Hom_\k(X^*, \Tc)$, under the left multiplication in $\C[G]$.
This picture will be quantized in this section.

  \begin{propn}
 \label{multiplicities}
  For every $V\in \Fin_q(\g)$ and for all $\xi, \eta \in \La^+_{\k}$,
  an  isomorphism
  $$\Hom_{U_q(\g)}(M_{\la,\eta},V\tp M_{\la,\xi})\simeq\Hom_\k(X_{\eta},V\tp X_{\xi})$$
    holds at all $q\in \Omega_t$.
\end{propn}
\begin{proof}
Corollary \ref{pseudo-parabolic-irreducible1} implies an equality $\Char (V\tp M_{\la,\xi})=\sum_{\eta \in I}\Char (M_{\la,\eta})$, where the summation is over weights in $V^{J_\xi^+}$ counted with multiplicities
(they parameterize an irreducible decomposition of $V\tp M_{\la,\xi}$). This equality
implies $\Char (V\tp X_\xi)=\sum_{\eta \in I }\Char (X_\eta)$, by the same corollary.
Therefore  the $\k$-module $\op_{\eta\in I}X_{\eta}$ is isomorphic to $V\tp X_{\xi}$
and the assertion follows.
\end{proof}
\noindent
It follows that for $t$ of finite order the set of exceptional $q$ can be chosen independent of $\xi,\eta$ and $V$.

We derive the following description of  isotypic components in $\Hom(M_{\la,\eta},M_{\la,\xi})$.
\begin{corollary}
\label{k-types}
  For every $V\in \Fin_q(\g)$ and for all $\xi, \eta \in \La^+_{\k}$,
  $$\Hom_{U_q(\g)}\bigl(V^*,\Hom(M_{\la,\eta},M_{\la,\xi})\bigr)\simeq\Hom_\k(X_\eta,V\tp X_\xi)$$
    holds at all $q\in \Omega_t$.
 \end{corollary}
\begin{proof}
Since $M_{\la,\xi}$ and $M_{\la,\eta}$ are irreducible along with their dual modules of lowest weight, equivariant maps  from $V^*$ to $\Hom(M_{\la,\eta},M_{\la,\xi})$
are in bijection with equivariant maps from $\Hom(M'_{\la,\xi},M'_{\la,\eta})$ to $V$,
for every $V\in \Fin_q(\g)$.
We have a version of  Proposition \ref{multiplicities} for dual modules of lowest weights:
$$
\Hom_{U_q(\g)}\bigl(\Hom(M'_{\la,\xi},M'_{\la,\eta}),V\bigr)\simeq \Hom_{U_q(\g)}\bigl(M_{\la,\eta}',V\tp M'_{\la,\xi}\bigr)\simeq   \Hom_\k(X_{\eta}^*,V\tp X_{\xi}^*).
$$
The rightmost term is isomorphic to $\Hom_\k(V^*\tp X_\xi,X_{\eta})\simeq \Hom_\k(X_{\eta},V^*\tp X_\xi)$ as $V^*\tp X_\xi$ is completely reducible over $\k$.
\end{proof}

For every locally finite $U_q(\g)$-module $A$ with finite dimensional isotypic components and for every pair of weights $\eta, \xi\in \La^+_{\k}$, define
$$
A^{(\xi,\eta)}= \Hom_{U_q(\g)}\bigl(M_{\la,\eta},A\tp M_{\la, \xi}\bigr)\simeq \Hom_{U_q(\g)}\bigl(M_{\la,\eta}\tp M_{\la, \xi}',A\bigr).
$$
In the case of  $\xi=0=\eta$ we will write $A^\k=A^{(0,0)}$.
It follows from the  right isomorphism that $A^{(\xi,\eta)}\simeq (\ker J^+_\xi\cap \ker J^-_\eta)[\eta-\xi]$ and therefore $(A^*)^{(\xi,\eta)}\simeq A^{(\eta,\xi)}$.
We  have a quasi-classical isomorphism  $A^{(\xi,\eta)}\simeq \Hom_\k\bigl(X^\eta,  X^\xi\tp A\bigr)$, by Proposition \ref{multiplicities}.
Note that   $A^+_0=\ker J^+_0$ is the sum of $A^{(0,\xi)}$ over all $\xi\in \La(A^+_{0})$ because such weights are highest for finite dimensional
$\k$-submodules in $V$ and therefore $\k$-dominant.
\begin{lemma}
\label{dec_of_inv}
For any pair of modules $V,W\in \Fin_q(\g)$,
$$
\begin{array}{ccccccc}
(V\tp W)^{\k}&\simeq  & \op_{\xi \in \La(W^+_0)} W^{(0,\xi)} \tp  V^{(\xi,0)}.
\end{array}
$$
\end{lemma}
\begin{proof}
This readily follows from complete reducibility:
$$
\Hom(M_\la, V\tp W \tp M_\la)\simeq  \op_{\xi\in \La(W^+_0)} W^{(0,\xi)}\tp \Hom(M_\la,V\tp M_{\la,\xi})\simeq  \op_{\xi\in \La(W^+_0)}W^{(0,\xi)}\tp V^{(\xi,0)},
$$
 and from the decomposition $W^+_0\simeq \op_{\xi\in \La(W^+_0)} W^{(0,\xi)}$.
\end{proof}

Now suppose that $A$ is an associative $U_q(\g)$-module algebra with multiplication $\cdot$.
Define a multiplication on $\Hom_{U_q(\g)}(M_\la,A\tp M_\la)\simeq A^{\k}$ by assigning the composition
$$
f_2\star f_1\colon M_\la\stackrel{f_1}{\longrightarrow} A\tp M_{\la}\stackrel{f_2}{\longrightarrow} A\tp (A\tp M_{\la})\stackrel{\cdot }{\longrightarrow} A\tp M_\la,
$$
to a pair of morphisms $f_1,f_2$. Clearly $\star$ is associative.

Furthermore, it extends to a right $A^\k$-action on $\Hom_{U_q(\g)}(M_\la,A\tp M_{\la,\xi})\simeq A^{(\xi,0)}$ by
$$
h\btl f\colon M_{\la}\stackrel{f}{\longrightarrow} A\tp M_{\la}\stackrel{\id\tp h}{\longrightarrow} A\tp A\tp M_{\la,\xi}\stackrel{\cdot }{\longrightarrow} A\tp M_{\la,\xi},
$$
and a left $A^\k$-action on $\Hom_{U_q(\g)}(M_{\la,\xi},A\tp M_\la)\simeq A^{(0,\xi)}$ by
$$
f\btr g\colon M_{\la,\xi}\stackrel{g}{\longrightarrow} A\tp M_{\la}\stackrel{f}{\longrightarrow} A\tp A\tp M_{\la}\stackrel{\cdot }{\longrightarrow} A\tp M_\la,
$$
Again associativity readily follows from associativity of composition and of the multiplication $\cdot$.

Denote by $\Sc^\nu\in U_q(\g_+)\tp U_q(\g_-)$ a lift of the inverse form of the irreducible module of highest
weight $\nu$.
The operations $\btr$ and $\btl$  can be written with the help of the inverse forms:
\be
h\btl  f&=&   (\Sc^\la_1\tr f)\cdot \Sc^\la_2\tr h),\quad h\in A^{(\xi,0)},\quad f\in A^\k,
\label{star_action_left}
\\
f \btr  g &=& (\Sc^\la_1\tr g)\cdot \Sc^\la_2\tr f), \quad f\in A^\k, \quad g\in A^{(0,\xi)}.
\label{star_action_right}
\ee

Explicit expressions for  $\Sc^\nu$ are known only for some special cases, e.g.
Verma modules \cite{M6} and base modules for quantum spheres \cite{M3}. With
the use of the relation between the inverse forms and extremal projectors (\ref{Shap_proj}), the
introduced operations can be presented in an explicit although more cumbersome form. Setting $\zt=\la+\xi$,
we write
\be
h\btl  f= (\>\cdot\>\tp \eps_\zt)\biggl( p_{\g}(0)\Bigl(p_{\g}^{-1}(\la)f\tp p_{\g}(0)
\bigl(p_{\g}^{-1}(\zt) h    \tp 1_{\la, \xi} \bigr)\Bigr)\biggr),\quad h\in A^{(\xi,0)},\quad f\in A^\k,
\nn
\ee
\be
f \btr  g = (\>\cdot\>\tp \eps_\la) \biggl( p_{\g}(0)\Bigl(p_{\g}^{-1}(\la)g\tp p_{\g}(0)
\bigl(p_{\g}^{-1}(\la) f\tp 1_{\la} \bigr)\Bigr)\biggr), \quad f\in A^\k, \quad g\in A^{(0,\xi)}.
\nn
\ee
The map $\ve_\zt$ is a linear functional that acts by  pairing with the highest vector via the contravariant form.

Next we study the classical limit of the introduced operations.
Let $ U_\hbar(\g)$ be the extension of $U_q(\g)$ over the ring of formal power series
in $\hbar =\ln q$.
Let   $N^\pm_\hbar\subset  U_\hbar(\g_\pm)$ be  $\C[[\hbar ]]$-submodules   that are $U_\hbar (\h)$-affine lifts of
$M_{\la}$ and $M_{\la}'$, respectively.

Consider a regular  $\hat U_\hbar(\h)$-submodule  in   $ \hat U_\hbar(\g)$, generated by $\gm^{-1}(N^-_\hbar ) N^+_\hbar$.
It contains the extremal twist $\Theta^\la$.

\begin{propn}
\label{class_lim_S}
The classical limit of $\Sc^\la$ is  $1\tp 1\in U(\k_+)\tp U(\k_-)$.
\end{propn}
\begin{proof}
Let $\theta_0\in N^-_0N^+_0$ denote the classical limit of $\theta_\la$.
Pick up a module $V\in \Fin_q(\g)$ and a pair of vectors $w,v\in \widehat{\!V^+_\la}$. In the classical limit,
one has $(\theta_\la^{-1} v,w)=(p_\g(\la) v,w)\to (p_\k v,w)$. Therefore, $\theta_\la$ tends to
identity on $V^{\k_+}$ for every $V$. Hence $\theta_0=1$ by Proposition \ref{ind-inv-pairing}.

Now regard $\Sc^\la_{21}$ as an element of $ N^+_\hbar N^-_\hbar$ under the linear isomorphism $M_{\la}'\tp M_{\la}\to N^+_\hbar N^-_\hbar \subset \hat U_\hbar(\g)$
facilitated by the triangular decomposition of $U_\hbar(\g)$.
Consider a linear automorphism $\hat U_\hbar(\g)\to \hat U_\hbar(\g)$ defined by
$$
e f h\mapsto \gm^{-1}(f) e h, \quad f\in U_\hbar(\g_-), \quad h\in U_\hbar(\h), \quad e\in U_\hbar(\g_+).
$$
This map takes $\Sc^\la_{21}$ to  $\theta_\la=\gm^{-1}(\Sc^\la_2)\Sc^\la_1$.
Therefore $\Sc^\la\to 1\tp 1$ in the classical limit as required.
\end{proof}
Now suppose that $A$ is a flat deformation of a $U(\g)$-module algebra $A_0$. Then $A^\k$ is a flat deformation (as a vector space) of the subspace
of $\k$ invariants $A^\k_0$.
\begin{corollary}
\label{star-prod-alg}
 The associative algebra $A^\k$ with multiplication $\star$ is a flat deformation of the algebra $A_0^\k$ with opposite
 multiplication.
 \end{corollary}
\begin{proof}
  The formula (\ref{star_action_left}) for $h\in \A^{0,0}$ turns to
$
h \star f=\Sc^\la_1 f \cdot \Sc^\la_2 h.
$
By Lemma \ref{class_lim_S}, $\Sc^\la \to 1\tp 1$ in the classical limit $q\to 1$, and the assertion follows.
\end{proof}

We take $\Tc$  with the $U_q(\g)$-action $\tr$ for the module-algebra $A$.
Endow the dual vector space $V^*$  of a $U_q(\g)$-module $V$ with a left action $x\diamond v^*=v^* \tl \gm(x)$. When applied to $\Tc$,
it gives rise to a left action, which is compatible with the opposite multiplication and commutes with $\tr$.

Note that $\star$-product is often defined as opposite to the one introduced above.
That version is equivariant with respect to  the right translation action $\tl$ while the present one respects $\diamond$.
\begin{thm}
\begin{enumerate}
   \item  The associative algebra $\Tc^\k$ with multiplication $\star$ is an equivariant flat deformation of $\C[O]$.
  \item   The $\Tc^\k$-modules $\Tc^{(0,\xi)}$ and $\Tc^{(\xi,0)}$ are $U_q(\g)$-equivariant flat deformations of the
   associated vector bundles on $O$ with fibers
  $X_\xi$ and its dual, respectively.
\end{enumerate}
\end{thm}
\begin{proof}
First of all note that the multiplication $\cdot$ is a star-product deformation of the classical multiplication on $\C[G]$, \cite{T}.
Furthermore, by Corollary \ref{k-types}, the $U_q(\g)$-module structure of $\Tc^{0,\xi}$ (resp. $\Tc^{\xi,0}$) under the action $\diamond$
is similar to the $U(\g)$-module structure of the associated vector bundles with fiber $X_\xi$ (resp. $X_\xi^*$).
So $\Tc^{0,\xi}$ and  $\Tc^{\xi,0}$ are flat deformations as $U_q(\g)$-modules. Finally,
  1) is a special case of Corollary \ref{star-prod-alg} for $A=\Tc^\k$ and
2) directly follows from Lemma \ref{class_lim_S}.
\end{proof}
Up to now we viewed sections of quantized vector bundles as equivariant linear maps from $\Hom_{U_q(\g)}(-,\Tc\tp -)$.
Let us give them an alternative interpretation in terms of  $\C$-linear maps between
objects from $\O_q(t)$. They are natural $U_q(\g)$-modules whose locally finite parts   also have a natural algebraic structure.

We endow the vector space $\Hom(A,B)$  between two $U_q(\g)$-modules $A$ and $B$ with a  left $U_q(\g)$-action
$(x\tr f)(a)=x^{(1)}f(\gm(x^{(2)}a)$.
\begin{propn}
\label{bundles_embedded}
Upon extension over the field  $\C(q)$, the algebra   $\Tc^\k$ is isomorphic to the  locally finite part of $\End(M_\la)$.
The isomorphism commutes with the action of $U_q(\g)$.
  The $\Tc^\k$-actions $\btr$ and $\btl$ go over to the  natural $\End(M_\la)$-actions on
  $\Hom(M_{\la, \xi},M_\la)$ and $\Hom(M_\la,M_{\la, \xi})$, respectively.
\end{propn}
\begin{proof}
  Define a map $\Tc^{(0,\xi)}\to \Hom(M_{\la,\xi},M_\la)$ for each $\xi\in \La^+_{\k}$  as follows.
  Every matrix element $g=v^*\tp v\in V^*\tp V^{(0,\xi)}\subset \Tc^{(0,\xi)}$
 goes to a linear map
\be
g\colon M_{\la,\xi} \supset x 1_{\la,\xi} \mapsto (v^*, x^{(1)}\Sc^\la_1v) \tp x^{(2)}\Sc^\la_21_\la \in M_{\la}, \quad x\in U_q(\g).
\nn\label{assoc_bund_hom}
\ee
This assignment is equivariant and its image in $\Hom(M_{\la,\xi},M_\la)$
  is the isotypic $V^*$-component, hence it is an isomorphism by Proposition \ref{multiplicities}.
  Furthermore,  for $f=w^*\tp w\in W^*\tp W^{(0,0)}$ we have
$$
\begin{array}{cccccc}
M_\la&\leftarrow& M_\la&\leftarrow &M_{\la,\xi}\\
&f\uparrow \hspace{10pt}&&g\uparrow\hspace{10pt}\\
&W^*& \tp &V^*
\end{array}
\quad= \quad
\begin{array}{cccccc}
M_\la &\longleftarrow &M_{\la,\xi}\\
&\hspace{25pt}\uparrow f\btr g\\
&\hspace{15pt}(V \tp W)^*
\end{array}
$$
The equality holds because the morphism in the right diagram is $\Sc^{\la}(g\tp f)$, and the tensor product is the multiplication $\cdot$ of matrix elements of representations constituting $\Tc$. This proves the statement with regard to $\btr$ (and $\star$ as a special case).
The case of $\btl$ is proved similarly.
\end{proof}
\noindent
Remark that for $t$ of finite order, the above assertion can be specialized at almost all $q\not =1$.

In the next section we describe quantized vector bundles as projective modules over $\Tc^\k$.
This models  local triviality of the classical vector bundles, \cite{S,Sw}.

\subsection{Quantum vector bundles as projective $\Tc^\k$-modules}

We saw in the previous section that the locally finite parts of the $U_q(\g)$-modules $\Hom(M_\la, M_{\la,\xi})$ and
$\Hom(M_{\la,\xi},M_\la)$   are isomorphic to $\Tc^{(\xi,0)}$ and, respectively, $\Tc^{(0,\xi)}$.
In particular, the locally finite part of $\End(M_\la)$ is a $U_q(\g)$-algebra, isomorphic to $\Tc^\k$.
The following result is obtained in \cite{JM}, Theorem 5.6 for quaternionic  projective plain but
 is valid in general.
\begin{thm}
The  $\Fin_q(\g)$-module category $\O_q(t)$  is equivalent to the category  of equivariant finitely generated projective right $\Tc^\k$-modules.
\end{thm}
\noindent
Similar statement holds upon replacement of right $\Tc^\k$-modules with left.

It follows that any invariant projector from $V\tp M_\la$ to an irreducible  submodule is a matrix with entries in the
locally finite part of $\End(M_\la)$, which coincides with $\Tc^\k$ if we allow for division by $q-1$.
The question is if
such a projector has classical limit. We did not check it for projective spaces  and even spheres  in \cite{M4,M5}  and we fill that gap here.\begin{lemma}
For each $W\in \Fin_q(\g)$ there is a quasi-classical isomorphism
$$
W\tp \Tc^\k\simeq \op_{[V]}V^*\tp (V\tp W)^{\k}
$$ of $U_q(\g)$-modules.
\label{rearrange_trivial}
\end{lemma}
\begin{proof}
Decomposing the left-hand side to isotypic components we write it as
$$
\op_{[Z]}(W\tp Z^*)\tp Z^{\k}\simeq \op_{[V]}V^*\tp \Bigl(\op_{[Z]}\Hom_{U_q(\g)}(V^*,W\tp Z^*)\tp Z^{\k}\Bigr).
$$
Replacing $\Hom_{U_q(\g)}(V^*,W\tp Z^*)$ with $\Hom_{U_q(\g)}(Z,V\tp W)$ we find the sum in the brackets equal to
$(V\tp W)^{\k}$ (a consequence of complete reducibility of $V\tp W$),  which implies the required isomorphism. It is obviously quasi-classical.
\end{proof}
We arrive at the main results of this section.
\begin{thm}
\label{quasi-classical projector}
The quantum vector bundles $\Tc^{(\xi,0)}$ and $\Tc^{(0,\xi)}$ are flat deformations as projective $\Tc^\k$-modules.
\end{thm}
\begin{proof}
It is sufficient to realize  $\Tc^{(\xi,0)}$ and $\Tc^{(0,\xi)}$ as direct summands in a free $\Tc^\k$-module
and show that  such a   decomposition is a deformation of the classical decomposition of the corresponding  trivial vector bundle.
We will do it only for $\Tc^{(0,\xi)}$ as the case of $\Tc^{(\xi,0)}$ is similar.

Pick up a module $W\in \Fin_q(\g)$ such that $\xi\in \La(W^+_0)$.
For almost all $q$ we have an equivariant diagram of isomorphisms of  $\Tc^\k$-modules from Proposition \ref{bundles_embedded}:
$$
\begin{array}{ccccccc}
\op_{\xi\in \La(W^+_0)}W^{(0,\xi)}\tp \Tc^{(0,\xi)}&\dashrightarrow & W^*\tp \Tc^\k\\
\downarrow& &\downarrow\\
\op_{\xi\in \La(W^+_0)}W^{(0,\xi)}\tp \Hom(M_{\la,\xi},M_\la)&\longrightarrow & \Hom(W\tp M_{\la}, M_{\la})\\
\end{array}
$$
where $\dashrightarrow$ is determined by the other arrows for almost all $q\not =1$.
On passing to the Peter-Weyl expansion $\Tc^{(0,\xi)}=\sum_{[V]} V^*\tp V^{(0,\xi)}$, this map operates by an  isomorphism on $\Hom$-s
$$
\begin{array}{ccccccc}
\op_{\xi \in \La(W^+_0)} W^{(0,\xi)} \tp  V^{(0,\xi)}&\stackrel{\simeq}{\dashrightarrow} & (V^*\tp W)^{\k}&\simeq & (W^*\tp V)^{\k}
\end{array}
$$
in each isotypic $V^*$-component thanks to Lemma \ref{rearrange_trivial}.
This is a consequence of  isomorphism $V^{(0,\xi)}\simeq (V^*)^{(\xi,0)}$  and Lemma \ref{dec_of_inv}.

All terms here are flat deformations of their classical counterparts by Proposition \ref{extremal_dim} and the isomorphism
is implemented via  extremal projectors which turn to projectors of $U(\k)$ as $q\to 1$.
In the classical limit, these isomorphisms turn to direct sum decomposition of the trivial vector bundle
$W\tp \Tc^\k$.
\end{proof}

In conclusion of the section, we present a direct quasi-classical construction of the algebra of $U_q(\g)$-intertwiners splitting trivial quantum vector bundles into direct sum of
sub-bundles.
Take $V\in \Fin_q(V)$ and let $\pi$ denote the representation homomorphism $U_q(\g)\to \End(V)$.
For a $U_q(\g)$-module  $\Ac$ with action $x\tp a\mapsto x.a$ we call a tensor $A\in \End(V)\tp \Ac$ invariant if $x. A=\pi\bigl(\gm(x^{(1)})\bigr)A\pi(x^{(2)})$.
In the case when $\Ac$ is a module algebra, invariant matrices form a subalgebra
of tensors  in $\End(V)\tp \Ac \rtimes U_q(\g)$ that commute with $(\pi\tp \id)\circ \Delta(x)$ for all $x\in U_q(\g)$.

Now consider an injective  linear map $\End(V)\to \End(V)\tp \Tc$, $A\mapsto  \pi(\Ru_1)TAT^{-1}\pi\bigl(\gm(\Ru_2) q^{2h_\rho}\bigr)$
(the Sweedler notation for R-matrix used).
The image of $A$ is an invariant matrix with respect to the action $\diamond$ on the entries.
By dimensional reasons, this map
delivers a linear isomorphism between $ \bigl(\End(V) \bigr)^{\k}$ and $\End(V)\tp \Tc^\k$ at almost all $q$ including the classical point.
With the $\star$-product
on $\Tc^\k$ the image  of $ \bigl(\End(V) \bigr)^{\k}$ is the algebra of invariant matrices separating quantum sub-bundles
in  $V\tp \Tc^\k$.
This algebra  is  quasi-classical by the mere construction and isomorphic to the subalgebra of classical $\k$-invariants
in $\End(V)$ by Proposition \ref{multiplicities}.
\appendix

\section{Induced modules and duality}
In this appendix we establish a fact that is used in the proof of Proposition \ref{class_lim_S}.
We think it should be a sort of classical but we have not found any reference
so we present it here.
Suppose $\g$ is a Lie algebra of a linear algebraic group and $\k \subset \g$ is a Lie  subalgebra of its closed subgroup.

Introduce an equivariant  pairing $\Ind_\k^\g \C\tp \C[G]^\k\to \C$,
induced by the pairing
\be
U(\g) \tp  \C[G]^\k\to \C, \quad u\tp \phi\mapsto (u.\phi)(1)
\nn
\ee
with $u\in U(\g)$ and $\phi\in \C[G]^\k$.
Presenting $\phi$ as a matrix element $w^*\tp w\in  W^*\tp W^\k$,
where $w \in  W^\k$, we rewrite the pairing as
$
u\tp (w^*\tp w)\mapsto w^*(u.w ).
$
\begin{propn}
\label{ind-inv-pairing}
Suppose there is a finite dimensional $\g$-module $V$ with a vector $v_0\in V^\k$ such that
$\dim \g v_0=\dim \g/\k$ and $v_0\not \in \g v_0$. Then the pairing
is non-degenerate with respect to both tensor factors.
\end{propn}
\begin{proof}
  It is straightforward that the pairing has no kernel in the right factor (that is a consequence
  that the Hopf pairing between $U(\g)$ and $\C[G]$ is non-degenerate).

  Pick up a basis in $\g$ such that its first $m$ elements  $u_1,\ldots, u_m$ span a subspace transversal to $\k$
  and put $v_i=u_i.v_0\in V$.
  Put $u^{\vec l}=u_1^{l_1}\ldots u_m^{l_m}$, where $\vec l=(l_1,\ldots, l_m)$ with $l_i\in \Z_+$, and denote $|\vec l|=\sum_{i=1}^{m}l_i$.
  Then ordered PBW monomials $u^{\vec l}$, form a basis
  in $\Ind_\k^\g \C$. Suppose that the element $u=\sum_{\vec l,j}c_{\vec l} u^{\vec l}$, where $c_{\vec l}\in \C$, is in the kernel
  of the pairing.
  The sum is finite, so let $n$ be the highest degree $|\vec l|$ of its $u^{\vec l}$.
  Take $w =v_0\tp \ldots \tp v_0 \subset V^{\tp n} $.
  As $u$ is in the kernel by assumption, one should have
  $
  \sum_{\vec l}c_{\vec l}u^{\vec l}.w=0.
  $
  Let us show  that it is impossible.

  Denote by  $\Sym$ the
  symmetrizing projector in tensor powers of $V_0$.
  Every such element with $|\vec l|=n$ produces $m!\Sym(v_1^{\tp l_1}\tp \ldots \tp v_m^{\tp l_m})+\ldots$, where
 the terms containing $\Sym(v_0\tp \ldots )$ are suppressed. Such terms are also resulted from
  $u^{\vec l}.w$ with $|\vec l|<n$.
  But $\Sym(v_1^{\tp l_1}\tp \ldots \tp v_m^{\tp l_m})$ are linearly independent and independent from
  the suppressed terms since  $\{v_i\}_{i=0}^{m}$ are independent. Therefore all $c_{\vec l}$ with $|\vec l|=n$ are zero.
  Descending induction on $n$  proves that the pairing has no kernel in $\Ind_\k^\g \C$.
\end{proof}
Here is a geometrical interpretation of the conditions of Proposition \ref{ind-inv-pairing}: the local homogeneous space
$(\g,\k)$ is realized as the orbit of the vector $v\in V$, and $v$ is transversal to the orbit.
These conditions are fulfilled for semi-simple conjugacy classes of algebraic groups. Such a class can be realized
as an orbit of a vector $v$ in some representation \cite{GW}, Theorem 11.1.13. Since the Cartan subalgebra is in
the stabilizer, $v$ carries  zero weight and is therefore transversal to $\g v$.

\vspace{20pt}
\noindent
\underline{\large \bf Acknowledgement.}

\vspace{5pt}
\noindent
This research was done at the Center of Pure Mathematics, MIPT.
It was was partly supported by Ministry of Science and Higher Education of
the Russian Federation, agreement no. 075-15-2022-289.

\subsection*{Declarations}
\subsubsection*{Funding}
The author received partial research support from   Ministry of Science and Higher Education of
the Russian Federation, agreement no. 075-15-2022-289.
\subsubsection*{Competing interests}
The author have no competing interests to declare that are relevant to the content of this article.

 \end{document}